\documentclass[11pt,reqno]{amsart}

\usepackage[utf8]{inputenc}
\usepackage[english]{babel}
\usepackage[a4paper]{geometry}

\usepackage{amssymb,amsmath,amsthm}
\usepackage[v2,tips]{xy}

\newcommand\blfootnote[1]{%
  \begingroup
  \renewcommand\thefootnote{}\footnote{#1}%
  \addtocounter{footnote}{-1}%
  \endgroup
}

\newtheorem{Thm}{Theorem}[section]
\newtheorem{Lem}[Thm]{Lemma}
\newtheorem{Prop}[Thm]{Proposition}

\newtheorem{Cor}[Thm]{Corollary}

\theoremstyle{definition}

\newtheorem{Rem}[Thm]{Remark}

\newtheorem*{Ack}{Acknowledgement}

\numberwithin{equation}{section}

\DeclareMathOperator*{\Ker}{\mathrm{Ker}}

\DeclareMathOperator*{\id}{\mathrm{id}}

\DeclareMathOperator*{\spanning}{\mathrm{span}}

\DeclareMathOperator*{\inv}{\mathrm{inv}}
\DeclareMathOperator*{\asy}{\mathrm{Asy}}

\newcommand{\mc}{\mathcal}

\title{Ring-theoretic (in)finiteness in reduced products of Banach algebras}
\author{Matthew Daws and Bence Horv\'{a}th}

\newcommand{\Addresses}{{% additional braces for segregating \footnotesize
  \bigskip
  \footnotesize
  
  \textit{Addresses}: \\
  Matthew Daws: \textsc{Jeremiah Horrocks Institute, University of Central Lancashire, Preston, PR1~2HE, United Kingdom} \\
  Bence Horv\'{a}th: \textsc{Institute of Mathematics, Czech Academy of Sciences, \v{Z}itn\'a 25, 115 67 Prague 1, Czech Republic}\par\nopagebreak
  \textit{E-mail addresses}: \\ 
  Matthew Daws: \texttt{matt.daws@cantab.net, mdaws@uclan.ac.uk} \\
  Bence Horv\'{a}th: \texttt{horvath@math.cas.cz, hotvath@gmail.com}
}}

\begin{document}

\begin{abstract}
We study ring-theoretic (in)finiteness properties -- such as \emph{Dedekind-finiteness} and \emph{proper infiniteness} -- of ultraproducts (and more generally, reduced products) of Banach algebras.

Whilst we characterise when an ultraproduct has these ring-theoretic properties in terms of its underlying sequence of algebras, we find that, contrary to the $C^*$-algebraic setting, it is not true in general that an ultraproduct has a ring-theoretic finiteness property if and only if ``ultrafilter many'' of the underlying sequence of algebras have the same property.  This might appear to violate the continuous model theoretic counterpart of {\L}o\'s's Theorem; the reason it does not is that for a general Banach algebra, the ring theoretic properties we consider cannot be verified by considering a bounded subset of the algebra of \emph{fixed} bound.  For Banach algebras, we construct counter-examples to show, for example, that each component Banach algebra can fail to be Dedekind-finite while the ultraproduct is Dedekind-finite, and we explain why such a counter-example is not possible for $C^*$-algebras.
Finally the related notion of having \textit{stable rank one} is also studied for ultraproducts.\end{abstract}

\maketitle

\subjclass{2010}{\textit{ Mathematics Subject Classification.} 46M07, 46H99 (primary); 16B99, 43A20 (secondary)

\blfootnote{\keywords{\textit{Key words and phrases.} Asymptotic sequence algebra, Banach algebra, bicyclic monoid, Cuntz semigroup, Dedekind-finite, {\L}o\'s's Theorem, properly infinite, stable rank one, semigroup algebra, ultraproduct}}

\section{Introduction}

The notion of \emph{central sequences} (with respect to a limit, or an ultrafilter limit) has long
been a key tool in the study and classification of von Neumann algebras (see
\cite[Section~3, Chapter~XIV]{tak3} as a starting point).  More recently, such ideas have also
become central to the classification of $C^*$-algebras, see \cite{kir} for example.  The study of
ultrapowers is intimately connected to model theory, and indeed continuous model theory
has recently been successfully applied to the study of von Neumann and $C^*$-algebras,
\cite{fhlrtvw, fhs}.
Furthermore, the analogue of the ultrapower where usual convergence is used, the asymptotic sequence
algebra, appears in the study of the set theory of $C^*$-algebras, \cite{far}.

In constructing ultraproducts or the asymptotic sequence algebra, we of course only consider bounded sequences.  Many properties of $C^*$-algebras, such as those considered in this paper -- being Dedekind-finite, being properly infinite -- can be verified for a $C^*$-algebra by only looking at bounded, in fact, norm one, elements.  In this paper, we show that this is not true for general Banach algebras, and that furthermore, this has implications for ultraproducts, and asymptotic sequence algebras, of Banach algebras.  For example, the property of being Dedekind-finite passes from each component Banach algebra to the asymptotic sequence algebra, but the converse is not true.  This is a manifestation of the fact that the very language of continuous model theory involves the use of \emph{bounded} metric spaces, \cite{bbh}.  We remark that the asymptotic sequence algebra appears to have not been systematically
studied for Banach algebras; we think that this is likely to prove to be a useful construction in
general Banach algebra theory.

Let us now be more precise.  The main objects of study in this paper are the following.  Let $(\mc A_n)$ be a sequence of Banach algebras and let $\ell^\infty(\mc A_n)$
be the Banach space of all bounded sequences $(a_n)$, where $a_n\in\mc A_n$ for each $n$, turned into
a Banach algebra with pointwise operations.  Similarly, let $c_0(\mc A_n)$ be the subspace of
sequences $(a_n)$ with $\lim_n \|a_n\|=0$.  Then $c_0(\mc A_n)$ is a closed, two-sided ideal of 
$\ell^\infty(\mc A_n)$ and in fact, when each $\mc A_n$ is unital,
$\ell^\infty(\mc A_n)$ is the \emph{multiplier algebra} of
$c_0(\mc A_n)$ (compare \cite[Section~13]{far} for example).  The \emph{asymptotic sequence algebra}
$\asy(\mc A_n)$ is the quotient algebra $\ell^\infty(\mc A_n) / c_0(\mc A_n)$.  Let $\mc U$ be a non-principal ultrafilter on $\mathbb N$ and let $c_{\mc U}(\mc A_n)$ be the closed, two-sided ideal of
$\ell^\infty(\mc A_n)$ formed of sequences $(a_n)$ with $\lim_{n\rightarrow\mc U} \|a_n\|=0$.
The quotient $\ell^\infty(\mc A_n) / c_{\mc U}(\mc A_n)$ is the \emph{ultraproduct}
$(\mc A_n)_{\mc U}$, see \cite{hein}.
As $c_0(\mc A_n) \subseteq c_{\mc U}(\mc A_n)$ the ultraproduct is ``smaller''
than the asymptotic sequence algebra, although for the questions we consider here there will be
little difference.  If $\mc A_n=\mc A$ for all $n$, we write $\asy(\mc A)$ and $(\mc A)_{\mc U}$,
the latter known as the \emph{ultrapower} of $\mc A$.

We shall denote by a capital letter
$A$, and so forth, an element $A=(a_n) \in \ell^\infty(\mc A_n)$.  Let $\pi:\ell^\infty(\mc A_n) \rightarrow \asy(\mc A_n)$ and $\pi_{\mc U}:\ell^\infty(\mc A_n) \rightarrow(\mc A_n)_{\mc U}$ be the
quotient maps; then
\[ \|\pi(A)\| = \limsup_{n\rightarrow\infty} \|a_n\|, \qquad
\|\pi_{\mc U}(A)\| = \lim_{n\rightarrow\mc U} \|a_n\|. \]
In particular, given any $a\in \asy(\mc A_n)$ we can always find $A=(a_n)\in\ell^\infty(\mc A_n)$
with $\pi(A)=a$ and $\|A\|=\sup_n \|a_n\| = \|a\|$, and similarly for $(\mc A_n)_{\mc U}$.
As our notation indicates, we only work with \emph{sequences} of algebras $(\mc A_n)$, and
not with general nets, though our results could be formulated in a more general setting.  We
always assume, then, that our ultrafilters are non-principal, which on a countable indexing set,
is equivalent to being \emph{countably-incomplete}, \cite[Section~1]{hein}.

An approach we could have taken to our overall presentation would have been to work with 
``reduced products'', compare \cite[Section~2.3]{gh}.  If $\mc F$ is merely a filter on $\mathbb N$
(and not necessarily an ultrafilter) we may still define $c_{\mc F}(\mc A_n)$ to be those sequences
$(a_n)\in \ell^\infty(\mc A_n)$ such, for each $\epsilon>0$, that $\{ n \in \mathbb{N} : \|a_n\|<\epsilon \}\in
\mc F$.  Then $c_{\mc F}(\mc A_n)$ is a closed, two-sided ideal, and so we may define
$(\mc A_n)_{\mc F} = \ell^\infty(\mc A_n) / c_{\mc F}(\mc A_n)$.
This definition agrees with the previous one if $\mc F$ does happen to be an ultrafilter.
Furthermore, if $\mc F$ is the Fr\'echet filter (so $A\in\mc F$ if and only if $\mathbb N
\setminus A$ is finite) then $c_{\mc F}(\mc A_n) = c_0(\mc A_n)$ and so $(\mc A_n)_{\mc F}
= \asy(\mc A_n)$.
Consequently, we could have structured all our statements and proofs to be about reduced products.
Instead, we felt that writing statements and proofs for $\asy(\mc A_n)$ improved the readability
(as we can work with ``normal'' convergence, and limits at $\infty$).  Once the structure of an
argument is understood, it is then easy to adapt it to work for (ultra)filter products.

As motivation, and to make explicit links with continuous logic, let us consider when $\asy(\mc A_n)$ or $(\mc A_n)_{\mc U}$ is unital.  The first
author considered the ultrapower case in \cite[Proposition~2.1]{daw}, showing that $(\mc A)_{\mc U}$
is unital if and only if $\mc A$ is, the proof using just techniques from functional analysis.  In Section~\ref{sec:prelin} below, we consider how to use the tools of continuous model theory to show, and improve upon, this result.  From this perspective, the bulk of the proof of \cite[Proposition~2.1]{daw} is taken up with showing that one can write a sentence $\varphi$ in the language of Banach algebras such that a Banach algebra $\mc A$ is unital if and only if the interpretation $\varphi^{\mc A}$ of $\varphi$ in $\mc A$ is zero.  This is not entirely trivial, as in continuous logic
the existential and universal quantifiers are replaced by supremum and infimum.  Once such a sentence has been found, {\L}o\'s's Theorem for continuous model theory shows that $\varphi^{\mc A}=0$ if and only if $\varphi^{(\mc A)_{\mc U}}=0$, hence immediately showing the result.  In Section~\ref{sec:prelin} below we show how to also obtain an analogous result for ultraproducts.

An interesting aspect of continuous logic is that the language of Banach algebras requires us to choose \emph{bounded} domains for any sentence or term we use.  These domains are typically chosen so that in an interpretation of the language, they are closed balls of varying radii.  For the statement ``$\mc A$ is unital'' this seems innocuous, as units always have norm one.  However, for an abstract Banach algebra (as opposed to $C^*$-algebras, or most concretely occurring Banach algebras) this is merely convention.  Indeed, if we allow Banach algebras to have units of norm greater than one, then we can find a sequence $(\mc A_n)$ of unital Banach algebras such that $(\mc A_n)_{\mc U}$ is not unital, see Proposition~\ref{prop:12}.
We shall explore this phenomena extensively for the algebraic properties we consider below.

Once units have been studied, it is natural to look at idempotents, that is, elements $p\in\mc A$ with $p^2=p$.  Two idempotents $p,q$ are \emph{equivalent}, written $p\sim q$, when there are $a,b\in\mc A$ with
$p=ab$ and $q=ba$.  This is indeed an equivalence relation, see Section~\ref{sec:prelin} for this.  We say that $p,q$ are \emph{orthogonal} if
$pq=0$ and $qp=0$.  When $\mc A$ is a $C^*$-algebra, it is more natural to take account of the star-structure, and to ask that our idempotents are self-adjoint, giving the notion of a \emph{projection}.  In Section~\ref{sec:prelin} we give a quick survey of the relation between idempotents and projections, showing in particular that for $C^*$-algebras, for the properties we consider, one can equivalently work with either idempotents, or projections.

A unital algebra $\mc A$ is \emph{properly infinite} if there are orthogonal idempotents
$p,q\in\mc A$ with $p\sim 1$ and $q\sim 1$.  $\mc A$ is \emph{Dedekind-finite} if $p\sim 1$ implies
$p=1$, and is otherwise \emph{Dedekind-infinite}.  In this paper, we study how these notions
interact with the ultraproduct, and, especially, the asymptotic sequence algebra constructions.

In Section~\ref{sec:dedekind_finite}, we show that if $(\mc A_n)$ is a sequence of Dedekind-finite Banach algebras then also
$\asy(\mc A_n)$ is Dedekind-finite.  The converse is not, in general, true, but we show that it is
under further conditions. The definition of $p \sim 1$ entails the existence of $a,b\in\mc A$ with $p=ab, ba=1$.  To write a sentence in the language of Banach algebras expressing this requires us, amongst other things, to give a bounded domain to work in: that is, to norm bound $a,b$.  For technical reasons, it is easier to work with the notion of being Dedekind-infinite, and we introduce two constants, $C_{\textrm{DI}}(\mc A)$ and $C'_{\textrm{DI}}(\mc A)$, which measure how large $\|a\|\|b\|$, or $\|p\|$ where $p$ is idempotent, need to be to show that $\mc A$ is Dedekind-infinite; see Section~\ref{sec:di} for precise statements.

We show in Proposition~\ref{diconstant} that if $(\mc A_n)$ is a sequence of Dedekind-infinite Banach algebras, but with unbounded $C_{\textrm{DI}}$ constants, then $\asy(\mc A_n)$ is Dedekind-finite.  We then show, by way of examples of weighted semigroup algebras, that such a sequence can exist.  By contrast, such an example cannot occur for $C^*$-algebras, see Corollary~\ref{cor:1}.

This exploration also raises the question of renormings.  For a unital Banach algebra, we can renorm so as to ensure the unit has norm 1.  We find examples which show that $C_{\textrm{DI}}(\mc A)$ and $C'_{\textrm{DI}}(\mc A)$ need not be comparable, meaning that we can show $\mc A$ is Dedekind infinite for an idempotent $p$ of small norm, but to verify that $p\sim 1$, that is, find $a,b\in \mc A$ with $p=ab, 1=ba$, we need to use $a,b$ of large norm.  We make some remarks on whether it is possible to renorm $\mc A$ to make $C_{\textrm{DI}}(\mc A)$ small, but ultimately leave this as an open question.

In Section~\ref{sec:prop_inf}, we perform a similar analysis for being properly infinite: if
$\asy(\mc A_n)$ is properly infinite then $\mc A_n$ is properly infinite for large enough $n$.
Again, the converse holds when we have uniform norm control on the norms of the elements of $\mc A$ showing that the idempotents $p,q$, involved in the definition of ``properly infinite'', are equivalent to $1$.  A more involved example of a weighted semigroup algebra provides a counter-example, Theorem~\ref{thm:pi_seq_not_pi}.  We again show that it is possible to control the norms of the idempotents without being able to control the norms of elements involved in verifying that $p\sim1, q\sim 1$, see Proposition~\ref{prop:9}.  We also investigate renorming questions. When we do have sufficient norm control, we completely characterise when $\asy(\mc A_n)$ is properly infinite, Proposition~\ref{piconstant_conv}. This is again achieved by introducing constants $C'_{\textrm{PI}}(\mc A)$ and $C_{\textrm{PI}}(\mc A)$, which measure how big the product of the norms $\|p \| \|q\|$ are which witness proper infiniteness, and respectively, the product of the norms of the elements implementing the equivalence of the idempotents. Along the way, we show that there is an isometric embedding of the inductive limit of Banach algebras $(\mc A_n)$ into $\asy(\mc A_n)$, and use this to draw conclusions about when the inductive limit is properly infinite.

Finally, in Section~\ref{sec:stab_rk1}, we consider the property of having stable rank one, which we view as a strengthening
of being Dedekind-finite.  Here $\asy(\mc A)$ having stable rank one implies the same for $\mc A$;
again the converse does not hold, which we show by way of a counter-example.  Two themes running
through our consideration of all three properties are ``lifting'' properties from $\asy(\mc A_n)$
to $\ell^\infty(\mc A_n)$, which we view as being interesting from a technical viewpoint; and also
the certain ``norm control'' considerations mentioned above.

The organisation of the paper is as follows.  In Section~\ref{sec:prelin} below, we first consider some aspects of continuous model theory, and in particular {\L}o\'s's Theorem, before giving some background results about idempotents, projections, and present some results about being ``close'' to an idempotent.  In the subsequent sections we study, respectively, being Dedekind-finite, being properly infinite,
and having stable rank one.  We finish the paper with some open problems.

\section{Preliminaries}\label{sec:prelin}

For us, a \emph{Banach algebra} $\mc A$ will always have a contractive product $\|ab\| \leq \|a\|
\|b\|$.  If $\mc A$ is unital then we assume that the unit $1$ has $\|1\|=1$.  These assumptions
can always be achieved by giving $\mc A$ an equivalent norm (see \emph{e.g.} \cite[Proposition~2.1.9]{dal} and the comment following it).  As many of the results in this paper depend upon exact norm control, and not just upon the equivalence class of the norm, we should be a little careful of renorming arguments.  The reader is pointed to Section~\ref{sec:renomings} and Propositions~\ref{prop:13} and~\ref{prop:14} below for a wider discussion.

Let us make some remarks about continuous logic. Again, we shall motivate this discussion by considering when $\mc A$ is unital.
That $\mc A$ is unital, with unit $e\in \mc A$, can be expressed in first-order logic as
$\exists\, e\in\mc A, \forall\, a\in\mc A, ae=ea=a$.  However, in continuous logic, we
have to use $\sup$ and $\inf$ in place of $\forall$ and $\exists$, and furthermore, we can only quantify over bounded subsets of $\mc A$.  
In fact, typically we define bounded domains $B_n$ which, when interpreted in $\mc A$ will be the closed ball of radius
$n\in\mathbb N$.  Consider the sentence
\[ \varphi = \inf_{e\in B_1} \sup_{a\in B_1} \max\big( \|ae-a\|, \|ea-a\| \big). \]
Arguing exactly as in the proof of \cite[Proposition~2.1]{daw}, we can show that a Banach algebra $\mc A$ is unital if and only if $\varphi^{\mc A} = 0$, that is, the interpretation of $\varphi$ in $\mc A$ is zero.

{\L}o\'s's Theorem for continuous model theory, \cite[Proposition~4.3]{fhs} or
\cite[Theorem~5.4]{bbh}, shows that $\mc A$ will satisfy this sentence if and only if any ultrapower
$(\mc A)_{\mc U}$ does.  This immediately shows \cite[Proposition~2.1]{daw}.  In effect, {\L}o\'s's
Theorem takes care of the ``ultraproduct bookkeeping'' for us.  In fact, we can use {\L}o\'s's Theorem to improve \cite[Proposition~2.1]{daw}, as follows.

\begin{Prop}\label{prop:unital_ultrapower}
Let $(\mc A_n)$ be a sequence of Banach algebras.  Then $(\mc A_n)_{\mc U}$
is unital if and only if there is $U\in\mc U$ with $\mc A_n$ unital for each $n\in U$.
\end{Prop}
\begin{proof}
Set $\mc B = (\mc A)_{\mc U}$ and suppose that $\mc B$ is unital.
By {\L}o\'s's Theorem we know that $\varphi^{\mc B} = \lim_{n\rightarrow\mc U} \varphi^{\mc A_n}$.  We are supposing that $\varphi^{\mc B} = 0$, so for $\epsilon \in (0, 1/4)$ to be chosen later, there is $U\in\mc U$ with $\varphi^{\mc A_n} < \epsilon$ for $n\in U$.  Given such an $n$ we can find $e\in \mc A_n$ with $\|e\|\leq 1$ and $\|ea-a\|, \|ae-a\| \leq \epsilon \|a\|$ for each $a\in\mc A_n$.  We now apply Corollary~\ref{cor:zem_for_non_unital} below, to find $p\in \mc A_n$ with $p^2=p$ and $\|e-p\|\leq f_1(\epsilon)$, where $f_1$ is defined by equation~\ref{eq:zem_const} below.

Then $\|pc-c\| \leq \|p-e\|\|c\| + \epsilon\|c\| \leq (f_1(\epsilon)+\epsilon)\|c\|$ for each $c\in\mc A_n$.  If we choose $\epsilon>0$ small enough then $f_1(\epsilon)+\epsilon<1$.  Given $a\in\mc A_n$ let $b=a-pa$ so that $pb=0$ because $p^2=p$.  If $b\not=0$ then we conclude that $\|b\| = \|pb-b\| < \|b\|$, a contradiction.  Hence $pa=a$ and similarly $ap=a$, and so $p$ is the unit of $\mc A_n$, as required.

The converse is clear.
\end{proof}

We shall show later in Proposition~\ref{prop:12} that if we allow a Banach algebra to have a unit of norm greater than 1 then the previous proposition is false.

Let us make two remarks about this treatment of $\mc A$ being unital.
The first remark is that we could work with the bounded domain $B_1$ above only because of our assumption that if a Banach algebra is unital, then its unit has norm 1.  However, this is only ensured by a renorming argument.  By contrast, for a $C^*$-algebra, a unit $e$ is self-adjoint (because units are unique) and an idempotent, and so is a projection, and hence has norm 1 automatically.

Renorming arguments are not problematic when considering ultrapowers (or $\asy(\mc A)$) but can become questionable when considering ultraproducts (or $\asy(\mc A_n)$ for a varying sequence $(\mc A_n)$).  We shall make further remarks later, see Section~\ref{sec:renomings} below.

The second remark is that the bulk of the argument of, for example, Proposition~\ref{prop:unital_ultrapower} is functional analytic: showing that if $\varphi^{\mc A}$ is small, then in fact $\mc A$ is unital.  The necessary Banach algebraic techniques for the properties we consider below are only more involved, and for this reason, we have chosen to present our results in a way which does not explicitly use continuous model theory.

\subsection{Background results about idempotents}

Let us quickly recall why $\sim$ is an equivalence relation on idempotents.  Only transitivity is
non-trivial.  Let $p\sim q$ and $q\sim r$, say with $p=ab, q=ba$ and $q=cd, r=dc$.  Then $p = p^2 =
abab = aqb = (ac)(db)$ and $(db)(ac) = dqc = dcdc = r^2 = r$ so $p\sim r$.

For a $C^*$-algebra $\mc A$, rather than considering idempotents, it is usual to consider
\emph{projections}, which are by definition self-adjoint idempotents: $p\in\mc A$ with $p=p^*=p^2$.
It is a fun exercise to show that an idempotent in a $C^*$-algebra is a projection if and only if it
has norm one.  Further, the natural equivalence of projections is that of \emph{Murray--von Neumann
equivalence}, which says that $p\approx q$ if and only if $p=v^*v, q=vv^*$ for some $v\in\mc A$
(which is necessarily a partial isometry).  The proof in the previous paragraph still works to show
that $\approx$ is an equivalence relation.  Finally, it is then usual to define \emph{properly
infinite} and \emph{Dedekind-(in)finite} using projections and Murray--von Neumann equivalence.

We claim that it does not matter, for a $C^*$-algebra, if we use our definitions or the
$C^*$-definitions.  This is folklore, but we have been unable to find a reference, so to aid the
reader, we give the argument.  For the next few results, we fix a unital $C^*$-algebra $\mc A$.

\begin{Lem}\label{lem:pr1}
Let $p\in\mc A$ be an idempotent.  There is a projection $q\in\mc A$ with $p\sim q$.  We can
arrange for $pq=q, qp=p$ or for $pq=p, qp=q$.
\end{Lem}
\begin{proof}
This is \cite[Exercise~3.11]{rll} or \cite[Proposition~3.2.10]{dal}. Set $a=1 +(p-p^*)^*(p-p^*) \geq 1$ so $a$ is invertible. Then $pa=pp^*p=ap$ and also $p^*a=ap^*$.  Set $q=pp^*a^{-1} = a^{-1}pp^*$. Then $q^2 = a^{-1} pp^*pp^*a^{-1} = a^{-1}app^*a^{-1} = q$ and $q=q^*$. Also $pq=q$ while $qp=a^{-1}pp^*p = a^{-1}ap=p$, hence $p\sim q$. If instead we set $r=p^*pa^{-1} = a^{-1}p^*p$ then $r^2=r=r^*$ while $pr=p$ and $rp=r$.
\end{proof}

\begin{Lem}\label{lem:pr2}
Let $p,q\in\mc A$ be projections with $p\sim q$.  Then $p\approx q$.
\end{Lem}
\begin{proof}
This is \cite[Proposition~3.2.10]{dal}.  Suppose $p=ab$ and $q=ba$.  Let $c=qbp$ thus $qc=c=cp$; as also $c = babab$ we have that $ac=p$ and $ca=q$, and consequently $cac=cp=c$.

Suppose $p\not=0$. We have that $p = p^*p = c^*a^*ac \leq \|a\|^2c^*c$. Working in $p\mc Ap$,
we see that $c^*c$ is invertible, so there is $d\in p\mc Ap$ with $d|c| = |c|d = p$.  Set
$u=cd$ so $u^*u=d^*|c|^2d = p$. Then $qu = cacd = cd = u$ and hence $uu^* = uu^*q = cdd^*c^*ca = cdd^*|c|^2a = cdp|c|a = cd|c|a = cpa = caca = q$. Thus $p\approx q$.

If $p=0, q\not=0$ then swap the roles of $p$ and $q$. If $p=q=0$ then clearly $p\approx q$.
\end{proof}

\begin{Prop}\label{prop:pr3}
$\mc A$ is properly infinite as a Banach algebra if and only if it is properly infinite as a
$C^*$-algebra.
\end{Prop}
\begin{proof}
If $\mc A$ is properly infinite as a Banach algebra then there are idempotents $p,q$ with $pq=qp=0$
and $p\sim 1, q\sim 1$.  By Lemma~\ref{lem:pr1} there are projections $p',q'$ with $pp'=p',
p'p=p$ and $qq'=q, q'q=q'$.  Then $1 \sim p \sim p'$ so $p'\approx 1$ by Lemma~\ref{lem:pr2},
and similarly $q'\approx 1$.
Also $q'p' = q'qpp' = 0$ and similarly $p'q'=0$, hence $\mc A$ is properly infinite as a
$C^*$-algebra.  The converse is clear.
\end{proof}

\begin{Prop}\label{prop:pr4}
$\mc A$ is Dedekind-finite as a Banach algebra if and only if it is Dedekind-finite as a
$C^*$-algebra.
\end{Prop}
\begin{proof}
If $\mc A$ is Dedekind-finite as a $C^*$-algebra, then let $p\in\mc A$ be an idempotent with
$p\sim 1$.  There is a projection $q\in\mc A$ with $q\sim p$.  Thus $q\sim p\sim 1$ so $q\sim 1$
hence $q\approx 1$, and consequently $q=1$.  We can arrange that $pq=q$, so as $q=1$,
we conclude that $1 = q = pq = p$.
Therefore $\mc A$ is Dedekind-finite as a Banach algebra.  The converse is clear.
\end{proof}

We now return to the general case.  For $M>0$ define
\begin{align}
f_M: \, [0,1/4) \rightarrow \mathbb{R}; \quad f_M(t) = (M + 1/2)((1-4t)^{-1/2}-1). \label{eq:zem_const}
\end{align}
It is clear that $f_M$ is a non-negative continuous function. Furthermore, $f_M \leq f_N$ 
when $N > M >0$.

The following lemma is well-known, it can be found for example in \cite[Lemma~2.1]{amnm2} without a proof.  For completeness we provide a (functional calculus argument based) proof.
In what follows, if $\mathcal{A}$ is an algebra, then $\inv(\mathcal{A})$ denotes the group of invertible elements of $\mathcal{A}$. If $a,b \in \mathcal{A}$, then the \textit{commutator} of $a$ and $b$ is $[a,b]:= ab -ba$.

\begin{Prop}\label{approxidemp}
Let $\mathcal{A}$ be a unital Banach algebra, and let $a \in \mathcal{A}$ be such that $\textstyle{\nu := \Vert a^2 - a \Vert < 1/4}$. Then there is an idempotent $p \in \mathcal{A}$ such that $\textstyle{\Vert p - a \Vert \leq f_{\Vert a \Vert}(\nu)}$ holds. Moreover, if $y \in \mathcal{A}$ is such that $[y,a]=0$ then $[y,p]=0$.
\end{Prop}	
\begin{proof}
As $\nu < 1/4$,  it follows that the series $\textstyle{\sum_{n=0}^{\infty} \binom{2n}{n} \nu^n}$ converges in $[0, \infty)$ with sum $(1-4 \nu)^{-1/2}$, consequently
\begin{align}
s:= \sum_{n=0}^{\infty} \binom{2n}{n} (a-a^2)^n
\end{align}
is absolutely convergent and therefore convergent in $\mc{A}$. Let us define $\textstyle{p:=(a-1/2)s + 1/2}$. Clearly, if $y \in \mc{A}$ is such that $[y,a]=0$ then $[y,s]=0$, and consequently $[y,p]=0$. We show that $p \in \mc{A}$ is an idempotent, which is equivalent to showing that $\textstyle{(2p-1)^2 = 1}$.

We first observe that by the Cauchy product formula
\begin{align}\label{noholcalc}
s^2 = \sum\limits_{n=0}^{\infty} \left( \sum\limits_{k=0}^n \dbinom{2k}{k} \dbinom{2(n-k)}{n-k} \right) (a-a^2)^n = \sum\limits_{n=0}^{\infty} 4^n (a-a^2)^n.
\end{align}	
Secondly, by $\nu < 1/4$ it follows that $\textstyle{1-4a+4a^2 \in \inv(\mc{A})}$ with $\textstyle{(1-4a+4a^2)^{-1}}= \textstyle{\sum_{n=0}^{\infty} (4(a-a^2))^n}$ by the Carl Neumann series. Thus $\textstyle{s^2 = (1-4a+4a^2)^{-1}}$ and consequently
$\textstyle{(2p-1)^2 =((2a-1)s)^2 = (2a-1)^2 s^2 = (4a^2 -4a +1)(1-4a+4a^2)^{-1}=1}$.
Moreover, we have that
\begin{align}
\Vert p-a \Vert &= \Vert (a- 1/2)s + 1/2 -a \Vert 
= \left\Vert (a- 1/2) (s-1) \right\Vert \notag \\
&\leq (\Vert a \Vert + 1/2) \Vert s-1 \Vert \notag
\leq (\Vert a \Vert + 1/2) \sum\limits_{n=1}^{\infty} \tbinom{2n}{n} \Vert a - a^2 \Vert^n \notag \\
& = (\Vert a \Vert + 1/2)((1-4 \nu)^{-1/2}-1) \notag 
 = f_{\Vert a \Vert} (\nu)
\end{align}
by the definition of $f_{\Vert a \Vert}$.
\end{proof}

\begin{Cor}\label{cor:zem_for_non_unital}
Let $\mc A$ be a Banach algebra, and let $a\in\mc A$ with $\nu:=\|a^2-a\|<1/4$ and $f_{\|a\|}(\nu) < 1$.  Then there is an idempotent $p\in\mc A$ with $\|p-a\|\leq f_{\|a\|}(\nu)$.
\end{Cor}
\begin{proof}
If $\mc A$ is not unital, consider the unitisation $\mc A^\sharp$ with
adjoined unit $1$.  So $\mc A^\sharp = \mc A \oplus \mathbb C1$ with norm
$\|a+\alpha 1\| = \|a\| + |\alpha|$.  By the preceding applied to $\mc A^\sharp$ there is $p=q+\alpha 1$ with $p^2=p$ and $\|p-a\|\leq f_{\|a\|}(\nu)$.  Thus $p^2 = q^2 + 2\alpha q + \alpha^2 1$ so $p^2=p$ implies that either $\alpha=0$ or $\alpha=1$.  If $\alpha=1$ then $\|p-a\| = \|q-a\| + |\alpha| \geq 1$ a contradiction. So $\alpha=0$ and $p=q$ is an idempotent.
\end{proof}

A related result is the following, which is also folklore, a stronger version of which was proved by Zem\'{a}nek in \cite[Lemma~3.1]{zemanek0}. For the convenience of the reader we give a self-contained elementary proof.

\begin{Lem}\label{zemanek}
	Let $\mathcal{A}$ be a unital Banach algebra, and let $p,q \in \mathcal{A}$ be idempotents with $\Vert p -q \Vert <1$. Then $p \sim q$.
\end{Lem}

\begin{proof}
	We first observe that $(p-q)^2$ commutes with $p$ and $q$. Indeed,
	\begin{align}
	p(p-q)^2 = p(p-pq-qp+q) = p-pqp = (p-pq -qp +q)p = (p-q)^2 p,
	\end{align}
	and similarly for $q$ and $(p-q)^2$. Now since $\Vert p-q \Vert <1$, clearly $\textstyle{\Vert (p-q)^2 \Vert <1}$ and thus as in the proof of Proposition \ref{approxidemp}, the series 
	\begin{align}
	d:= \sum\limits_{n=0}^{\infty} \dbinom{2n}{n} \frac{(p-q)^{2n}}{4^n}
	\end{align}
	converges (absolutely) in $\mathcal{A}$, and $d$ commutes with $p$ and $q$.
	
	Again, as in the proof of Proposition \ref{approxidemp} we conclude $\textstyle{d^2 = (1-(p-q)^2)^{-1}}$. Another easy calculation shows
	\begin{align}
	(p+q-1)^2 = p+pq-p+qp+q-q-p-q+1 = 1-p +pq +qp -q = 1- (p-q)^2.
	\end{align}
	We define $c:= d(p+q-1)$. Since $p$ and $q$ commute with $d$, it follows that
	\begin{align}
	c^2 =d^2(p+q-1)^2 = d^2 (1-(p-q)^2) = 1.
	\end{align}
	We are now ready to show that $p \sim q$. To see this, we first observe
	\begin{align}
	pc &= pd(p+q-1) = dp(p+q-1) = d(p+pq-p) = dpq \notag \\
	cq &= d(p+q-1)q = d(pq+q-q) = dpq
	\end{align}
	and thus $pc=cq$. Consequently $(pc)(cp) = pc^2p=p$ and $(cp)(pc)= cpc = c^2q =q$ follow, concluding the proof.
\end{proof}

\begin{Rem}\label{rem:zemanek}
From this proof, we see that $p=ab$ and $q=ba$ for $a,b\in \mathcal{A}$, where $a=pc, b=cp$, and $c=d(p+q-1)$, where $d$ is given by a power series which yields the norm estimate
\[ \|d\| \leq \big(1-\|(p-q)^2\|\big)^{-1/2} \leq \big(1-\|p-q\|^2\big)^{-1/2}. \]
Thus $\|a\|\|b\| \leq \|p\|^2 \|p+q-1\|^2 (1-\|p-q\|^2)^{-1}
\leq \|p\|^2 (\|2p-1\|+ \|p-q\|)^2 (1-\|p-q\|^2)^{-1}$.
\end{Rem}

\section{Dedekind-finiteness}\label{sec:dedekind_finite}

In this section, we show that if $\mc A_n$ is Dedekind-finite for each $n$, then also
$\asy(\mc A_n)$ is Dedekind-finite.  The converse is not true without some form of ``norm control'',
and we provide a counter-example in the Banach algebra case, while also clarifying why the
converse does hold for $C^*$-algebras.

\subsection{When the sequence consists of Dedekind-finite algebras}

In the following proof, for clarity, given a sequence $(\mc A_n)$ of unital algebras, we write $1_n$
for the unit of $\mc A_n$.

\begin{Thm}\label{dfpositive}
Let $(\mathcal{A}_n)$ be a sequence of Dedekind-finite Banach algebras. Then $\asy(\mathcal{A}_n)$ is Dedekind-finite.
\end{Thm}
\begin{proof}
Let $p \in \asy(\mathcal{A}_n)$ be an idempotent.  Choose $X =(x_n) \in
\ell^\infty(\mc A_n)$ with $\pi(X)=p$, so that $\pi(X^2) = \pi(X)^2 = p^2 = p = \pi(X)$, or
equivalently, $X-X^2 \in c_0(\mc A_n)$.  Let us introduce $\nu_n := \Vert x_n - x^2_n \Vert$ for every $n \in \mathbb{N}$, then $\lim_{n \rightarrow \infty} \nu_n =0$. In particular, there is $N \in \mathbb{N}$ such that for every $n \geq N$ we have $\nu_n < 1/8$. In view of Proposition~\ref{approxidemp}, for every $n \geq N$ there is an idempotent $p'_n \in \mathcal{A}_n$ with $\Vert x_n - p'_n \Vert \leq f_{\Vert x_n \Vert}(\nu_n) \leq f_{\Vert X \Vert}(\nu_n) \leq f_{\Vert X \Vert}(1/8)$. By continuity of $f_{\Vert X \Vert}$, it follows that $\lim_{n \geq N} f_{\Vert X \Vert}(\nu_n) =0$; consequently $\lim_{n \geq N} \Vert x_n - p'_n \Vert =0$. 
For every $n \in \mathbb{N}$ we define
\begin{align}
p_n := \left\{
\begin{array}{l l}
p'_n & \quad \text{if  } n \geq N \\
0 & \quad \text{otherwise.} \\
\end{array} \right.
\end{align}
Since $\Vert p'_n \Vert \leq \Vert p'_n -x_n \Vert + \Vert x_n \Vert \leq f_{\Vert X \Vert}(1/8) + \Vert X \Vert$ for all $n \geq N$, it follows that $P:=(p_n)$ is an idempotent in $\ell^\infty(\mc A_n)$.  We observe that $p= \pi(P)$ by $\lim_{n \geq N} \Vert x_n - p'_n \Vert =0$.

Now suppose further that $p \sim 1$, so there exist $a,b \in \asy(\mathcal{A}_n)$ such that $1=ab$ and $p=ba$.  There are $A =(a_n)$, $B=(b_n) \in \ell^\infty(\mc A_n)$ such that $a= \pi(A)$ and $b = \pi(B)$, consequently $\lim_{n \rightarrow \infty} \Vert 1_n - a_n b_n \Vert =0$ and $\lim_{n \rightarrow \infty} \Vert p_n - b_n a_n \Vert =0$.

Now let $\delta \in (0,1)$ be such that
\begin{align}
\Vert A \Vert \Vert B \Vert \delta/(1- \delta) +2 \delta <1.
\end{align}
Let $M \geq N$ be such that for all $n \geq M$ the inequality $\Vert 1_n - a_n b_n \Vert < \delta$ holds, then $u_n:= a_n b_n \in \inv(\mathcal{A}_n)$ with $\Vert 1_n - u_n^{-1} \Vert < \delta/(1- \delta)$. For every $n \geq M$, let $q_n := b_n u_n^{-1} a_n$, then $q_n \in \mathcal{A}_n$ is an idempotent with $q_n \sim 1_n$. Since $\mathcal{A}_n$ is Dedekind-finite, it follows for all $n \geq M$ that $q_n = 1_n$.

We need to show that $p=1$ holds, which is equivalent to showing $\lim_{n \rightarrow \infty} \Vert 1_n - p_n \Vert =0$. Since $1_n - p_n \in \mathcal{A}_n$ is an idempotent for all $n \in \mathbb{N}$, it is enough to show that eventually $\Vert 1_n - p_n \Vert < 1$, compare
Remark~\ref{rem:c_di_defn} below.  Let $K \geq M$ be such that for every $n \geq K$ 
\begin{align}
\Vert x_n - b_n a_n \Vert < \delta, \quad \Vert x_n - p'_n \Vert < \delta.
\end{align}
Then for every $n \geq K$ we have $p_n = p'_n$ and $1_n = q_n$, thus
\begin{align}
\Vert 1_n - p'_n \Vert &= \Vert q_n - p'_n \Vert \notag \\
&= \Vert b_n u_n^{-1} a_n - p'_n \Vert \notag \\
&\leq \Vert b_n u_n^{-1} a_n - b_n a_n \Vert + \Vert b_n a_n - x_n \Vert + \Vert x_n - p'_n \Vert \notag \\
&\leq \Vert b_n \Vert \Vert u_n^{-1} -1_n \Vert \Vert a_n \Vert + \Vert b_n a_n - x_n \Vert + \Vert x_n - p'_n \Vert \notag \\
&\leq \Vert A \Vert \Vert B \Vert \delta / (1- \delta) + 2 \delta \notag \\
& < 1.
\end{align}
This concludes the proof.
\end{proof}

\begin{Rem}
We note that the proof above gives some extra information.  Indeed, the first paragraph of the proof of Theorem \ref{dfpositive} shows precisely that idempotents from the asymptotic sequence algebra $\asy(\mc A_n)$ can always be lifted to idempotents in $\ell^{\infty}(\mc A_n)$.
\end{Rem}

\subsection{When the asymptotic sequence algebra is Dedekind-finite.}
\label{sec:di}

In this section we demonstrate that the converse of Theorem~\ref{dfpositive} holds for certain specific cases; but in general it does not hold, which we show by way of a counter-example.

In order to do this, let us introduce the following auxiliary quantity.  For a unital Banach algebra $\mathcal{A}$, we define
\begin{align}
C_{\text{DI}}(\mathcal{A}):= \inf \lbrace \Vert a \Vert \Vert b \Vert : \; a,b \in \mathcal{A} \text{ and } ab= 1 \text{ and } ba \neq 1 \rbrace.
\end{align}
We may also introduce the auxiliary constant
\begin{align}
C'_{\text{DI}}(\mathcal{A}) := \inf \lbrace \Vert p \Vert : \; p \in \mathcal{A} , p^2= p, p\sim 1, \text{ and } p \neq 1 \rbrace.
\end{align}
If $\mathcal{A}$ is Dedekind-infinite then $1 \leq C_{\text{DI}}(\mathcal{A}) < + \infty$, otherwise (that is, if $\mathcal{A}$ is Dedekind-finite) $C_{\text{DI}}(\mathcal{A}) = + \infty$, and similarly
for $C'_{\text{DI}}$.
Clearly $C'_{\text{DI}}(\mathcal{A}) \leq C_{\text{DI}}(\mathcal{A})$, but otherwise these quantities
are not comparable, see Proposition~\ref{prop:2}.  As a definition, perhaps
$C'_{\text{DI}}(\mathcal{A})$ seems more natural, but we shall see that $C_{\text{DI}}(\mathcal{A})$
is more useful in constructions.

\begin{Rem}\label{rem:c_di_defn}
If $p\in \mathcal{A}$ is an idempotent, then $p^n=p$ for any $n\in\mathbb N$, and so $\|p\|\geq 1$ or $p=0$.  Suppose we have $a,b\in \mathcal{A}$ with $ab=1$.  Then $p=ba$ is an idempotent, and hence so is $1-p$, hence either $p=1$ or $\|1-p\|\geq 1$.  So in the definition of $C_{\text{DI}}$ we also have that $ba$ is far away from $1$.
\end{Rem}

\begin{Prop}\label{diconstant}
Let $(\mathcal{A}_n)$ be a sequence of unital Banach algebras such that $C_{\text{DI}}(\mathcal{A}_n) \rightarrow + \infty$. Then $\asy (\mathcal{A}_n)$ is Dedekind-finite.
\end{Prop}

\begin{proof}
Let $A= (a_n), B=(b_n) \in \ell^\infty(\mc A_n)$ be such that $\pi(A)\pi(B) = 1$.  We wish to prove
that $\pi(B)\pi(A)=1$.  By the assumption we can take $N' \in \mathbb{N}$ such that $C_{\text{DI}}(\mathcal{A}_n) > 2 \Vert A \Vert \Vert B \Vert + 1$ whenever $n \geq N'$. Let us define $u_n := a_n b_n$ for every $n \in \mathbb{N}$. Since $\lim_{n \rightarrow \infty} \Vert 1_n - u_n \Vert =0$, we can pick $N \geq N'$ such that $\Vert 1_n -u_n \Vert < (2 \Vert A \Vert \Vert B \Vert +1)^{-1}$ for all $n \geq N$. Then $u_n \in \inv(\mathcal{A}_n)$ and $\Vert 1_n - u_n^{-1} \Vert \leq (2 \Vert A \Vert \Vert B \Vert)^{-1}$ and $\Vert u_n^{-1} \Vert \leq (2 \Vert A \Vert \Vert B \Vert + 1)(2 \Vert A \Vert \Vert B \Vert)^{-1}$ for all $n \geq N$.
	
Let us define $q_n := b_n u_n^{-1} a_n$ for every $n \geq N$.  As $a_n (b_nu_n^{-1}) = 1_n$ it follows that $q_n \in \mathcal{A}_n$ is an idempotent with $q_n \sim 1_n$.  Clearly, either $q_n =1_n$ or $q_n \neq 1_n$. Assume towards a contradiction that there is some $m\geq N$ with $q_m \neq 1_m$. Then
\begin{align}
C_{\text{DI}}(\mathcal{A}_m) \leq \Vert b_m u_m^{-1} \Vert \Vert a_m \Vert \leq \Vert b_m \Vert \Vert u_m^{-1} \Vert \Vert a_m \Vert \leq (2 \Vert A \Vert \Vert B \Vert +1)/2,
\end{align}
which is impossible. Hence $q_n = 1_n$ for all $n \geq N$.
	
From $\pi(A)\pi(B) = 1$ it follows that $\pi(B) \pi(A) \in \asy(\mathcal{A}_n)$ is an idempotent, which is equivalent to saying that $\lim_{n \rightarrow \infty} \Vert b_n a_n b_n a_n - b_n a_n \Vert =0$. Let $M' \geq N $ be such that $\nu_n := \Vert b_n a_n b_n a_n - b_n a_n \Vert < 1/8$ whenever $n \geq M'$. By Proposition~\ref{approxidemp} there is an idempotent $p'_n \in \mathcal{A}_n$ such that $\Vert b_n a_n - p_n' \Vert \leq f_{\Vert A \Vert \Vert B \Vert}(\nu_n)$, whenever $n \geq M'$. Let $p_n := p'_n$ if $n \geq M'$ and $p_n := 1_n$ if $n < M'$. Then $P:=(p_n) \in\ell^\infty(\mathcal{A}_n)$ is an idempotent with $\pi(B) \pi(A) = \pi(P)$.
	
Let $M \geq M'$ be such that $\Vert b_n a_n -p_n \Vert < 1/2$ for all $n \geq M$. Thus we obtain
\begin{align}
\Vert 1_n - p_n \Vert &= \Vert q_n -p_n \Vert \leq \Vert q_n - b_n a_n \Vert + \Vert b_n a_n - p_n \Vert \notag \\
&\leq \Vert b_n \Vert \Vert u_n^{-1} - 1_n \Vert \Vert a_n \Vert + \Vert b_n a_n - p_n \Vert \notag \\
&< \frac{\Vert A \Vert \Vert B \Vert}{2 \Vert A \Vert \Vert B \Vert} + 1/2 \notag \\
&=1
\end{align}
and hence, by Remark~\ref{rem:c_di_defn}, $1_n = p_n$ for all $n \geq M$.  This yields $\pi(B) \pi(A) = \pi(P) = \pi(1)$, showing that $\asy (\mathcal{A}_n)$ is Dedekind-finite.
\end{proof}

In particular as $C'_{\text{DI}}(\mathcal{A}) \leq C_{\text{DI}}(\mathcal{A})$ and $C'_{\text{DI}}(\mathcal{A}) < + \infty$ if and only if $C_{\text{DI}}(\mathcal{A}) < + \infty$, we immediately obtain the following.
	
\begin{Cor}\label{altdiconstant}
	Let $(\mathcal{A}_n)$ be a sequence of unital Banach algebras such that $C'_{\text{DI}}(\mathcal{A}_n) \rightarrow + \infty$. Then $\asy (\mathcal{A}_n)$ is Dedekind-finite.
\end{Cor}

\subsection{A counter-example}
We shall now construct Banach algebras which satisfy the hypothesises of Proposition~\ref{diconstant}.

Let $I$ be a non-empty set. For a fixed $s \in I$, $\delta_s$ denotes the function
\begin{align}
\delta_s(t) := \left\{
\begin{array}{l l}
1 & \quad \text{if  } t=s, \\
0 & \quad \text{if } t \neq s. \\
\end{array} \right.
\end{align}
Let $\nu: \, I \rightarrow (0, + \infty)$ be a function. We define
\begin{align}
\ell^1(I, \nu) := \left\lbrace f: \; I \rightarrow \mathbb{C} : \; \Vert f \Vert_{\nu} := \sum\limits_{s \in I} \vert f(s) \vert \nu(s) < + \infty \right\rbrace.
\end{align}
We have that $\ell^1(I, \nu) = \overline{\spanning \{ \delta_s : \, s \in I\}}^{\| \cdot \|_{\nu}}$, and further we can write $f= \sum_{s \in I} f(s) \delta_s$ for each $f \in \ell^1(I, \nu)$ where the sum converges in the norm $\| \cdot \|_{\nu}$.  It is easy to see that $(\ell^1(I, \nu), \Vert \cdot \Vert_{\nu})$ is a Banach space. In line with the general convention, we will simply write $\ell^1(I, \nu)$ for this Banach space, and $\ell^1(I)$ whenever $\nu =1$.

When $I$ is a monoid, there is a canonical way of turning $\ell^1(I, \nu)$ into a unital Banach algebra.
Slightly more generally, let $S$ be a semigroup.  Let $\omega: \, S \rightarrow (0, + \infty)$ be a \textit{weight} on $S$, that is, we require $\omega(st) \leq \omega(s) \omega(t)$ to hold for all $s,t \in S$. In addition, when $S$ is a monoid with multiplicative identity $e \in S$ then we also require $\omega(e) = 1$.  We remark that any weight is equivalent to one satisfying this normalisation
condition.  We define the usual \emph{convolution product} on $\ell^1(S, \omega)$ by
\begin{align}
(f \ast g)(r) := \sum\limits_{st=r} f(s)g(t) \quad (f,g \in \ell^1(S,\omega), \, r \in S),
\end{align}
then $(\ell^1(S,\omega), \ast)$ is a Banach algebra (one uses the condition on the weight to show
that the norm is submultiplicative.)   When $S$ is a monoid, $(\ell^1(S,\omega), \ast)$ becomes a unital Banach algebra with unit $\delta_e$, clearly $\Vert \delta_e \Vert_\omega =1$ holds.

In what follows, we shall mostly be interested in weights $\omega$ with $\omega(s)\geq 1$ for all $s$.
Notice then that $\ell^1(S,\omega)$ becomes a (in general, not closed) subalgebra of $\ell^1(S)$.

\begin{Prop}\label{altpredfcounterex}
Let $S$ be a monoid with unit $e \in S$ and let $\omega: \, S \rightarrow [1, + \infty)$ be a weight on $S$. Let $p \in (\ell^1(S, \omega), \ast)$ be a non-zero idempotent such that $p \neq \delta_e$. Then
\begin{align}
\Vert p \Vert_{\omega} \geq \frac{1}{2} \inf \left\lbrace \omega(s): \, s \in S,
s\not=e \right\rbrace.
\end{align}
\end{Prop}

\begin{proof}
Since $(\ell^1(S, \omega), \ast)$ is a subalgebra of $(\ell^1(S), \ast)$, we have that $p \in (\ell^1(S), \ast)$. Assume first that $p(e) \neq 0$.  We claim that then $\Vert \delta_e -  (p(e))^{-1} p \Vert \geq 1$. Indeed otherwise $\Vert \delta_e -  (p(e))^{-1} p \Vert < 1$ and thus $(p(e))^{-1}p$, and so $p$, are invertible in $(\ell^1(S), \ast)$, which is impossible as $p$ is an idempotent different from $\delta_e$. Consequently
\begin{align}
\vert p(e) \vert \leq \vert p(e) \vert \left\Vert \delta_e -  \frac{1}{p(e)} p \right\Vert = \sum\limits_{s \in S \setminus \lbrace e \rbrace} \vert p(s) \vert.
\end{align}
If $p(e) =0$ then the above inequality obviously holds. As $p$ is an idempotent, we have $1 \leq \Vert p \Vert$, and this yields
\begin{align}
1 \leq \Vert p \Vert = \vert p(e) \vert + \sum\limits_{s \in S \setminus \lbrace e \rbrace} \vert p(s) \vert \leq 2  \left( \sum\limits_{s \in S \setminus \lbrace e \rbrace} \vert p(s) \vert \right).
\end{align}
From this we conclude
\begin{align}
\Vert p \Vert_{\omega} &= \sum\limits_{s \in S} \vert p(s) \vert \omega(s) \geq \sum\limits_{s \in S \setminus \lbrace e \rbrace} \vert p(s) \vert \omega(s) \notag \\
&\geq \inf\limits_{s \in S \setminus \lbrace e \rbrace} \omega(s) \sum\limits_{s \in S \setminus \lbrace e \rbrace} \vert p(s) \vert \notag
\geq \frac{1}{2} \inf\limits_{s \in S \setminus \lbrace e \rbrace} \omega(s).
\end{align}
\end{proof}

In what follows $BC$ denotes the bicyclic monoid, which is the free monoid generated by elements $p,q$
subject to the single relation that $pq=e$:
\begin{align}
BC = \langle p, q : \; pq= e \rangle.
\end{align}
Fix $n \in \mathbb{N}$. Let $\omega_n : \, BC \rightarrow [1, + \infty)$ be the weight on $BC$ defined as $\omega_n(e) = 1$ and $\omega_n(s) = n$ for $s \in BC \setminus \lbrace e \rbrace$.

\begin{Thm}\label{thm:1}
Let $\mathcal{A}_n := (\ell^1(BC, \omega_n), \ast)$ for every $n \in \mathbb{N}$.
Then $(\mathcal{A}_n)$ is a sequence of Dedekind-infinite Banach algebras such that $\asy(\mathcal{A}_n)$ is Dedekind-finite. 
\end{Thm}
\begin{proof}
For any $n\in\mathbb N$, work in $\mc A_n$, and consider $h:= \delta_q \ast \delta_p$.  Then $h$ is an idempotent with $h \sim \delta_e$ and $h \neq \delta_e$.  Indeed, $\delta_p \ast \delta_q = \delta _{pq} = \delta_e$ and $\delta_q \ast \delta_p = \delta_{qp} \neq \delta_e$.  This in particular shows that $\mathcal{A}_n$ is Dedekind-infinite.

Now let $h \in \mathcal{A}_n$ be an arbitrary idempotent such that $h \sim \delta_e$ and $h \neq \delta_e$.  We observe that Proposition~\ref{altpredfcounterex} yields $\Vert h \Vert_{\omega_n} \geq n/2$, and consequently $C'_{\text{DI}}(\mathcal{A}_n) \geq n / 2$.  In view of Corollary~\ref{altdiconstant} the Banach algebra $\asy(\mathcal{A}_n)$ is Dedekind-finite.
\end{proof}

In particular, this shows that the converse of Theorem~\ref{dfpositive} does not hold in general.
As we can vary finitely many of the $\mc A_n$ without changing $\asy(\mc A_n)$, by using the
contrapositive, we can alternatively state Theorem~\ref{dfpositive} as: if $\asy(\mc A_n)$ is
Dedekind-infinite, then infinitely many of the $\mc A_n$ are Dedekind-infinite.
If we add in control of $C_{\text{DI}}$ then we obtain a converse.

\begin{Prop}\label{prop:1}
Let $(\mathcal{A}_n)$ be a sequence of unital Banach algebras such that there exists $K>0$ and an increasing sequence $(n_k)$ in $\mathbb{N}$ with $C_{\text{DI}}(\mathcal{A}_{n_k}) \leq K$ for each $k \in \mathbb{N}$.  Then $\asy(\mathcal{A}_n)$ is Dedekind-infinite.
\end{Prop}
\begin{proof}
By assumption, for each $k \in \mathbb{N}$ we can find $a_{n_k},b_{n_k}\in \mathcal{A}_{n_k}$ with $\|a_{n_k}\| \|b_{n_k}\| \leq K+1$, say, and $a_{n_k} b_{n_k} = 1_{n_k}$ while $b_{n_k} a_{n_k}\not=1_{n_k}$.  We can rescale and suppose that $\|a_{n_k}\| = \|b_{n_k}\|$.  For $n \in \mathbb{N}$ not in the sequence, define $a_n=b_n=1_n$.
Then $A=(a_n) \in \ell^\infty(\mathcal{A}_n)$, and similarly for $B=(b_n)$, and clearly $\pi(A)\pi(B) = 1$.  By Remark~\ref{rem:c_di_defn}, we have that $\|b_{n_k} a_{n_k} - 1_{n_k} \|\geq 1$ for each $k$, and so $\pi(B)\pi(A)\not=1$.  Thus $\asy(\mathcal{A}_n)$ is Dedekind-infinite.
\end{proof}

Furthermore, under certain conditions, we do obtain a direct converse to Theorem~\ref{dfpositive}.

\begin{Cor}\label{cor:1}
Let $(\mathcal{A}_n)$ be a sequence of unital Banach algebras such that $\asy(\mathcal{A}_n)$ is Dedekind-finite.  Moreover, suppose that one of the following two conditions hold:
\begin{enumerate}
\item $\mathcal{A}_n = \mathcal{A}_m$ for every $n,m \in \mathbb{N}$;
\item $\mathcal{A}_n$ is a $C^*$-algebra for each $n \in \mathbb{N}$.
\end{enumerate}
Then there is $N\in\mathbb N$ such that $\mathcal{A}_n$ is Dedekind-finite for $n\geq N$.
\end{Cor}

\begin{proof}
If $\mathcal{A}_n= \mathcal{A}$ for each $n$, then if $\mathcal{A}$ is Dedekind-infinite, then $C_{\text{DI}}(\mathcal{A})<\infty$ and so Proposition~\ref{prop:1} shows that $\asy(\mathcal{A}_n)$ is Dedekind-infinite, a contradiction.

Consider now a $C^*$-algebra $\mathcal{B}$.  By Proposition~\ref{prop:pr4}, we know that $\mc B$ is
Dedekind-finite if and only if it is Dedekind-finite in the $C^*$-algebraic sense, that is, if
$u$ is a partial isometry with $u^*u=1$ then $uu^*=1$.  It follows that $C_{\text{DI}}(\mathcal{B})=1$ or $+\infty$.  Thus, if each $\mathcal{A}_n$ is a $C^*$-algebra, and $\asy(\mathcal{A}_n)$ is Dedekind-finite, then Proposition~\ref{prop:1} shows that $C_{DI}(\mathcal{A}_n)=+\infty$ for all but finitely many $n$.  That is, eventually $\mathcal{A}_n$ is Dedekind-finite, as claimed.
\end{proof}

In the proof of Proposition~\ref{prop:1}, it seemed important to work with $C_{\text{DI}}$ and not
$C'_{\text{DI}}$.  In fact this is necessary, as we now show.
For the following, we need some simple combinatorics of
the monoid $BC$.  Any element of $BC$ can be written as a reduced word in the generators $p,q$,
which is necessarily of the form $q^\alpha p^\beta$ with $\alpha,\beta\in\mathbb N_0$.  The
multiplication law is that
\[ (q^\alpha p^\beta)(q^\gamma p^\delta) = \begin{cases}
q^\alpha p^{\delta+\beta-\gamma}  &\text{if } \beta\geq\gamma, \\
q^{\alpha-\beta+\gamma} p^\delta &\text{if } \beta\leq\gamma. \end{cases} \]
From this, it is easy to see that the set of idempotents in $BC$ is $BC_I=\{q^\alpha p^\alpha:
\alpha\geq 0\}$.  The following are also easy to see:
\begin{enumerate}
\item $(q^\alpha p^\alpha)(q^\beta p^\beta) = q^\gamma p^\gamma$ where $\gamma=\max(\alpha,\beta)$;
\item if $s\in BC_I, t\not\in BC_I$ then $st, ts\not\in BC_I$.
\end{enumerate}

By point (1) we see that $BC_I$ is a sub-semigroup of $BC$, and so we may consider $\ell^1(BC_I)$,
which can be identified with a closed subalgebra of $\ell^1(BC)$.

\begin{Lem}\label{lem:six}
Let $S:= \lbrace s_n \rbrace$ be a countable semigroup consisting of idempotents such that $s_n s_m = s_{\max(n,m)}$ for every $n,m \in \mathbb{N}$. Then $f \in \ell^1(S)$ is an idempotent
if and only if $f(s_n) \in \{-1,0,1\}$ for all $n$, there is an $n_0 \geq 0$ so that $f(s_n)=0$ for
$n\geq n_0$, and for all $m\geq 0$ we have that $\sum_{n=0}^m f(s_n) \in \{0,1\}$.
\end{Lem}
\begin{proof}
Firstly, let $f$ be an idempotent.  To ease notation, let $e_n := \delta_{s_n}$ and $f(n):= f(s_n)$, so that
$f$ has the expansion $f = \sum_{n\geq 0} f(n) e_n$.  Then
\[ f^2 = \sum_{n,m\geq 0} f(n) f(m) e_{\max(n,m)} = f
\quad\implies\quad
\sum_{\max(n,m)=t} f(n) f(m) = f(t) \qquad (t\geq 0). \]
Hence $f(0)^2 = f(0)$ and so $f(0)\in\{0,1\}$.  We now use strong induction: suppose that
$f(n) \in \{-1,0,1\}$ for each $n \leq N$ and that $\sum_{n=0}^m f(n) \in \{0,1\}$ for
each $m\leq N$.  Then
\[ f(N+1) = 2f(N+1) \sum_{n=0}^N f(n) + f(N+1)^2. \]
Either $\sum_{n=0}^N f(n) =0$, in which case $f(N+1)\in\{0,1\}$, and so
$\sum_{n=0}^{N+1} f(n) \in \{0,1\}$; or alternatively $\sum_{n=0}^N f(n) = 1$
in which case $0 = f(N+1) + f(N+1)^2$ so $f(N+1)\in \{0,-1\}$ and so
$\sum_{n=0}^{N+1} f(n) \in \{0,1\}$.
Finally, as $\sum_{n=0}^\infty f(n)$ must converge, we must have that $f(n)=0$ eventually.

We now consider the converse.  Given such an $f$, we have that
\[ \sum_{\max(n,m)=t} f(n) f(m) = 2f(t)\sum_{n<t} f(n) + f(t)^2 \quad (t \geq 1). \]
Either $\sum_{n<t} f(n)=0$, in which case, as $\sum_{n \leq t} f(n) \in \{0,1\}$ we must have
that $f(t)\in\{0,1\}$, and so $f(t)^2 = f(t)$ as required; or $\sum_{n<t} f(n)=1$, in which case,
as $\sum_{n \leq t} f(n) \in \{0,1\}$ we must have
that $f(t)\in\{-1,0\}$, and so $2f(t)\sum_{n<t} f(n) + f(t)^2
= f(t)(2 + f(t)) = f(t)$ for either $f(t)=0$ or $f(t)=-1$, as required.
\end{proof}

\begin{Prop}\label{prop:8}
Let $\omega$ be a weight on $BC$ such that $\omega\geq 1$ and $\omega(s)\geq N$ for
each $s\not\in BC_I$.  Then $C_{\text{DI}}(\ell^1(BC,\omega)) \geq (N/86)^{1/3}$.
\end{Prop}
\begin{proof}
Let $\mc A = \ell^1(BC, \omega)$, and suppose that $K > C_{\text{DI}}(\mc A)$, so
that $K>1$, and we can find $f,g\in\mc A$ such that $f*g=\delta_e$ and $g*f\not=\delta_e$,
and with $\|f\|_\omega \|g\|_\omega \leq K$.  As $\omega\geq 1$ we can regard $\mc A$ as a
(possibly not closed) subalgebra of $\ell^1(BC)$.

Let $h :=g*f$, then $h$ is an idempotent.  Let $k\in\ell^1(BC_I)$ be the restriction of $h$ onto $BC_I$.
Working in $\ell^1(BC)$, we notice that
\[ \|h\|_\omega = \sum_{s\in BC_I} |h(s)|\omega(s) + \sum_{s\not\in BC_I} |h(s)|\omega(s)
\geq N \sum_{s\not\in BC_I} |h(s)| = N \|h-k\|, \]
by our assumption on the weight.  We now observe that, because $h^2=h$,
\begin{align*} \|k^2-k\| &= \|k(k-h) + (k-h)h + h - k \| \\
&\leq (\|k\| + \|h\| + 1) \|k-h\|
\leq \frac{\|h\|_\omega}{N} (2\|h\|_\omega+1). 
\end{align*}
Let this quantity be $\nu$, and suppose that $\nu < 1/4$. Working in $\ell^1(BC_I)$, by Proposition~\ref{approxidemp}
there is an idempotent $k' \in \ell^1(BC_I)$ with
\[ \|k-k'\| \leq (\|k\| + 1/2)( (1-4\nu)^{-1/2} -1 ). \]
Let this quantity be $\nu'$.  Working in $\ell^1(BC)$, we have that
\[ \| h-k' \| \leq \|h-k\| + \|k-k'\|
\leq \nu' + \frac{\|h\|_\omega}{N}. \]
Let this quantity be $\epsilon$.

Assume that $\epsilon<1$.  By Remark~\ref{rem:c_di_defn}, as $h\not=\delta_e$,
we have that $\|\delta_e-h\|\geq 1$, and so $k'\not=\delta_e$.  

Let $e_n = \delta_{q^np^n}$ for all $n \in \mathbb{N}_0$ and write
\[ k' = \sum_{n\geq 0} k'(n) e_n. \]
As $k'\in\ell^1(BC_I)$
is an idempotent, we hence conclude, using Lemma~\ref{lem:six}, that
there is some $t\geq 1$ with $k'(t)=\pm1$ and so certainly $|h(t)\mp 1|<\epsilon$.

Write
\[ f = \sum_{n\geq 0} f(n) e_n + f' = f'' + f' \]
say, where $f'$ is supported off $BC_I$, and similarly for $g$.  Then
\[ f*g = \sum_{n,m\geq 0} f(n) g(m) e_{\max(n,m)}
+ f' * g'' + f'' * g' + f' * g'. \]
Notice that by the points above Lemma \ref{lem:six}, we have that $f'*g'' + f''*g'$ is supported off $BC_I$.
Write $(f'g')_n$ for the coefficient of $e_n$ in the expansion of $f'*g'$, and similarly
for $g'*f'$.  That $f*g=\delta_e$ means that
\begin{equation} \sum_{ \max(n,m)=t } f(n)g(m) + (f'g')_t = 0 \qquad (t\geq 1).
\label{eq:1}
\end{equation}

Also, $|h(t) \mp 1| < \epsilon$ and that $g' \ast f'' + g'' \ast f'$ is supported off $BC_I$ means that
\[ \Big|\sum_{ \max(n,m)=t } f(n)g(m) + (g'f')_t \mp 1\Big| < \epsilon \qquad (t \geq 1). \]
Using also (\ref{eq:1}) we see that $|-(f'g')_t + (g'f')_t \mp 1| < \epsilon$.
Thus $|(g'f')_t - (f'g')_t|>1-\epsilon$,
and so $\|f'\| \|g'\| \geq \frac12(1-\epsilon)$.  It follows that
\[ K \geq \|f\|_\omega \|g\|_\omega \geq N^2 \|f'\| \|g'\| \geq \frac{N^2}{2}(1-\epsilon). \]

Choose $\epsilon_0\in (0,1)$, and suppose that $43 K^3 \leq N \epsilon_0$, so certainly
$K^2 < N/15$.  As $\|h\|_\omega \leq K$, we see that $\nu \leq K(2K+1) / N \leq 3K^2/N < 1/5$,
and so $(1-4\nu)^{-1/2}-1 \leq 7\nu$, so that
\[ \nu' \leq (K+1/2)((1-4\nu)^{-1/2}-1)
\leq 14K\nu \leq \frac{42K^3}{N}. \]
Finally, $\epsilon = \nu' + \|h\|_w/N \leq \frac{42K^3+K}{N} \leq \frac{43K^3}{N}
\leq \epsilon_0$.  We conclude that if $43K^3 \leq N \epsilon_0$ then
$\frac12 N^2 (1-\epsilon_0) \leq K$.

Now, $K>C_{\text{DI}}(\mc A)$ is arbitrary, so take $\epsilon_0=1/2$, for example, to
conclude that if $C_{\text{DI}}(\mc A) < (N/86)^{1/3}$, then we can choose a suitable
$K \leq (N\epsilon_0/43)^{1/3}$, and so conclude that
\[ \frac{N^2}{2}(1-\epsilon_0) = \frac{N^2}{4} \leq K \leq (N/86)^{1/3}
\quad\implies\quad
N^5 \leq \frac{64}{86}, \]
which is a contradiction, as required.
\end{proof}

\begin{Prop}\label{prop:2}
For any $n\in\mathbb N$ we can find a unital Banach algebra $\mc{A}_n$ with
$C'_{\text{DI}}(\mc{A}_n)=1$ yet $C_{\text{DI}}(\mc{A}_n) \geq n$.
\end{Prop}
\begin{proof}
Given $n$, pick $N$ with $N \geq 86n^3$.
Define $\omega_n$ on $BC$ by $\omega_n(e)=1$, $\omega_n(qp)=1$, $\omega_n(s)=N$ otherwise.
Let $X = \{e, qp\} \subseteq BC$.  We now prove that $\omega_n$ is a weight, for which we need to
show that $\omega_n(st)\leq \omega_n(s) \omega_n(t)$ for all $s,t \in BC$.  This can only fail if
we can find $s,t$ such that $\omega_n(st)=N$ yet $\omega_n(s)=\omega_n(t)=1$, which is
impossible as $X$ is a sub-semigroup of $BC$.

Now consider $\mc{A}_n:=\ell^1(BC, \omega_n)$, and let $h =\delta_{qp}\in\mc{A}_n$.  Then $\|h\|_{\omega_n} = 1$ and as before, $h\sim \delta_e$, so $C'_{\text{DI}}(\mc{A}_n) \leq 1$, hence $C'_{\text{DI}}(\mc{A}_n) = 1$.  However, by Proposition~\ref{prop:8}, $C_{\text{DI}}(\mc A_n)
\geq (N/86)^{1/3} \geq n$, as required.
\end{proof}

Using this proposition, we thus obtain a sequence $(\mc A_n)$ of
Dedekind-infinite Banach algebras, with $C'_{\text{DI}}(\mc A_n)=1$ for each $n$, but with
$C_{\text{DI}}(\mc A_n) \rightarrow \infty$.  By Proposition~\ref{diconstant} we find that
$\asy(\mc A_n)$ is Dedekind-finite.  We conclude that in Proposition~\ref{prop:1} we
cannot replace $C_{\text{DI}}$ with $C'_{\text{DI}}$.

\begin{Rem}
Notice also that in this way, we obtain a sequence $(\mc A_n)$ of Banach algebras, and
idempotents $p_n\in\mc A_n$ such that each $p_n$ is equivalent to $1_n$, but the (equivalence
class of the) sequence $P=(p_n)$ is not equivalent to $1$ in $\asy(\mc A_n)$.
\end{Rem}

\begin{Rem}\label{rem:7}
The weights we have so far constructed have the property that $\ell^1(BC,\omega_n)$ is
isomorphic (just not isometric) with $\ell^1(BC)$, for the formal identity map.
However, we can easily construct examples which do not have this property.

Viewing $BC$ as the set of reduced words of the form 
$s = q^\alpha p^\beta$ we define the \emph{word length} as $\ell(s) = \alpha+\beta$.  It is easy to
see that $\ell:BC\rightarrow\mathbb{N}_0$ is sub-additive: $\ell(st) \leq \ell(s)+\ell(t)$.
As such, $\omega_x(s) = \exp(x\ell(s))$ defines a weight for every $x\geq 0$.
Furthermore, $\ell^1(BC, \omega_x)$ is not (naturally) isomorphic to $\ell^1(BC)$ for any $x>0$.

We now adapt the construction in the proof of Proposition~\ref{prop:2}.
Define, for $x\geq 0$,
\[ \omega_x(s) = \begin{cases} 1 &\text{if } s=e\text{ or } s=qp, \\
\exp(x\ell(s)) &\text{otherwise.} \end{cases} \]
This is a weight by the same argument as in the proof of Proposition~\ref{prop:2}.  As before,
with $\mc A_x = \ell^1(BC, \omega_x)$, we have $C'_{\text{DI}}(\mc A_x)=1$ while
$C_{\text{DI}}(\mc A_x) \geq (e^x/86)^{1/3}$.
\end{Rem}

\subsection{Renormings}\label{sec:renomings}

In Section~\ref{sec:prelin} we mentioned that any unital Banach algebra can be renormed so as
to make the norm of the unit be $1$.  Let us explore this further.  Firstly, we could take a more
``permissive'' definition of a unital Banach algebra: a complete normed algebra, with a contractive
product, and an element $e\in \mc A$ with $ea=ae=a$ for each $a\in\mc A$.  That is, we do not
assume that $\|e\|=1$.  Notice that if $\mc A$ is a Banach algebra with a unit of norm one,
then for $m\geq 1$ we can define $\|a\|_m = m\|a\|$ for $a\in\mc A$.  Then $\|\cdot\|_m$ is
an equivalent norm on $\mc A$, and as $\|ab\|_m = m\|ab\| \leq m\|a\| \|b\| \leq m^2\|a\|\|b\|
=\|a\|_m \|b\|_m$, we have a contractive product, but of course now $\|e\|_m=m$.

\begin{Prop}\label{prop:12}
With this definition of a unital Banach algebra, there exists a sequence $(\mc A_n)$ of unital
Banach algebras such that $\asy(\mc A_n)$ is not unital.
\end{Prop}
\begin{proof}
Fix a unital Banach algebra $\mc A$ with a norm one unit $e$.  For each $n\in\mathbb N$ let
$\mc A_n = (\mc A, \|\cdot\|_n)$ with $\|\cdot\|_n$ defined as above.  Towards a contradiction,
suppose that $f\in\asy(\mc A_n)$ is a unit, say $f=\pi(F)$ with $F=(f_n)\in\ell^\infty(\mc A_n)$,
so that
\[ 0= \limsup_{n\rightarrow\infty} \| a_nf_n - a_n \|_n 
= \limsup_{n\rightarrow\infty} n \| a_nf_n - a_n \| \]
for each $A=(a_n)\in\ell^\infty(\mc A_n)$.
Notice that $(a_n)\in\ell^\infty(\mc A_n)$ exactly
when there is $K>0$ with $\|a_n\| \leq Kn^{-1}$ for each $n$. 
Let $b_n:= e/n$ for all $n \in \mathbb{N}$. Then $\| b_n \| = 1/n$ and $\|b_n \|_n =1$, hence $B:=(b_n) \in \ell^{\infty}(\mathcal{A}_n)$.

 We hence see that
\[ 0= \limsup_{n\rightarrow\infty} n \| b_n f_n - b_n \| = \limsup_{n\rightarrow\infty} \| f_n - e \|. 
\]

However, as $\|f_n\| \leq Kn^{-1}$ for some fixed $K$, we see that
$\| e -f_n \| \geq \|e\| - \|f_n\| \geq 1 - K/n$ which is a contradiction.
\end{proof}

We thus see that it pays to be careful about ``implicit renormings'', when considering the sort
of questions we are asking.  We are also now lead to wonder how the constants $C_{\text{DI}}
(\mc A)$ and $C'_{\text{DI}}(\mc A)$ behave under renormings.  The following is a result in this
direction; we remark that we revert to our standing assumption that unital Banach algebras have
contractive multiplications and units of norm one.

\begin{Prop}\label{prop:10}
Let $\mc A$ be a Dedekind-infinite Banach algebra.  There is an equivalent norm $\|\cdot\|_0$ on $\mc A$
such that $(\mc A, \|\cdot\|_0)$ is a unital Banach algebra, and
$C'_{\text{DI}}(\mc A, \|\cdot\|_0)=1$.
\end{Prop}
\begin{proof}
Let $p \in \mc A$ be an idempotent with $p \sim 1$ and $p \neq 1$. The set $S:= \lbrace 1, p \rbrace$ is clearly a bounded sub-semigroup of $\mathcal{A}$.  Hence by
\cite[Proposition~2.1.9]{dal} there is a submultiplicative norm $\| \cdot \|_0$ equivalent to $\| \cdot \|$ on $\mathcal{A}$ such that $(\mathcal{A}, \| \cdot \|_0)$ is  Banach algebra and $\| s \|_0 \leq 1$ for every $s \in S$. In particular $\|1\|_0 \leq 1$ and $\|p \|_0 \leq 1$, hence $\|1\|_0 = 1$ and $\|p \|_0 = 1$. Consequently $C_{DI}'(\mathcal{A}, \| \cdot \|_0) = 1$.

For the benefit of the reader, we remark that \cite[Proposition~2.1.9]{dal} is proved as
follows.  Firstly we define $\nu(a) = \sup\{ \|a\|, \|sa\| : s\in S\}$ and then notice that $\nu$
is an equivalent norm on $\mc A$ with bounded product, such that $S$ is $\nu$-bounded, and such
that $\nu(sa)\leq \nu(a)$ for $a\in\mc A, s\in S$.  Then let $E$ be the unconditional unitisation
of $(\mc A, \nu)$, and finally let $\|\cdot\|_0$ be the norm induced on $\mc A$ by the
left-regular representation of $\mc A$ on $E$.
\end{proof}

\begin{Rem}
For each $K \geq 1$, we have examples of $\mc A$ where $C'_{\text{DI}}(\mc A) \geq K$, see the proof of
Theorem~\ref{thm:1} for example.  Thus, in the previous proposition,
while the original norm $\|\cdot\|$ and the new norm $\|\cdot\|_0$ are equivalent,
there is not some absolute constant $C>0$ (independent of $\mc A$) with
$C^{-1}\|\cdot\| \leq \|\cdot\|_0 \leq C\|\cdot\|$.
\end{Rem}

Notice that in the above proof, we actually showed something more: given any unital Banach algebra $\mc A$ with an idempotent $p \in \mc A$ such that $p\sim 1$ and $p\not=1$, we can find an equivalent norm $\|\cdot\|_0$ on $\mc A$ with $\|p\|_0=1$, hence $p$ witnesses that $C'_{DI}(\mc A, \| \cdot \|_0) =1$. The following shows that a similar statement does not hold for $C_{\text{DI}}$.

\begin{Prop}\label{prop:11}
For each $K>0$ there is a Banach algebra $\mc A$ and $a,b\in\mc A$ with $ab=1, ba\not=1$,
and such that, if $\|\cdot\|_0$ is any equivalent norm on $\mc A$, then
$\|a\|_0 \|b\|_0 \geq K$.
\end{Prop}
\begin{proof}
We use the same weight $\omega_x$ as in Remark~\ref{rem:7}, and set $\mc A_x =
\ell^1(BC,\omega_x)$.  Let $\|\cdot\|_0$ be an equivalent norm on $\mc A_x$, so there is $m>0$
with $m^{-1} \|f\|_{\omega_x} \leq \|f\|_0 \leq m \|f\|_{\omega_x}$ for all $f\in\mc A_x$.
Let $s = q^\alpha p^\beta \in BC$, with say $\alpha > \beta$.  Then
\[ s^2 = q^{2\alpha-\beta} p^\beta,\quad
s^3 = q^{3\alpha-2\beta} p^\beta, \quad
\cdots, \quad s^n = q^{n\alpha-(n-1)\beta} p^\beta. \]
It follows that
\begin{align*} & m^{-1} \exp(x(n\alpha-(n-2)\beta)) = m^{-1}\|\delta_{s^n}\|_{\omega_x}
\leq \|\delta_{s^n}\|_0 \leq \|\delta_s\|_0^n \\
\implies\qquad &
\frac{-\log(m)}{n} + x \frac{n\alpha-(n-2)\beta}{n} \leq \log \|\delta_s\|_0 \\
\implies\qquad &  \log \|\delta_s\|_0 \geq \liminf_{n\rightarrow\infty} \left( 
\frac{-\log(m)}{n} + x \frac{n\alpha-(n-2)\beta}{n} \right) = x( \alpha - \beta ).
\end{align*}
Thus $\|\delta_s\|_0 \geq e^{x(\alpha-\beta)}$.  The same argument applies in the case when
$\alpha < \beta$.

As in the proof of Proposition~\ref{prop:8}, set $a = \delta_p$ and $b = \delta_q$, so that
$ab=1$ and $ba\not=1$.  We have just shown that $\|a\|_0 \geq e^x$ and $\|b\|_0 \geq e^x$,
which completes the proof.
\end{proof}

We have been unable to decide if it is possible to renorm $\mc A_x$ so as to get 
$C_{\text{DI}}(\mc A_x, \|\cdot\|_0) = 1$ (or just be smaller than some absolute constant).
The difficulty is that, given the previous proposition, we need consider the possibility of
some other elements $c,d\in\mc A_x$ with $cd=1$ and $dc\not=1$, while also considering an
arbitrary renorming.

\subsection{For ultraproducts}
All of these results hold for ultraproducts with suitable modifications.
Let us start by indicating how to give a ``bare-hands'' proof.

\begin{Thm}\label{dfposultra}
Let $(\mc A_n)$ be a sequence of Banach algebras, and let $\mc U$ be an ultrafilter
such that $\{ n \in \mathbb{N} : \mc A_n \text{ is Dedekind-finite} \} \in\mc U$.  Then $(\mc A_n)_{\mc U}$
is Dedekind-finite.
\end{Thm}
\begin{proof}
We simply adapt the proof of Theorem~\ref{dfpositive}.  Firstly, we find that there is
$U\in\mc U$ with $n\in U\implies \nu_n =\|x_n-x_n^2\| < 1/8$.  This allows us to find
$P=(p_n)$ so that $\pi(P) = p \in (\mc A_n)_{\mc U}$.  Then we find $U'\subseteq U$ with
$U'\in\mc U$ and $u_n^{-1}$ existing for $n\in U'$.  But moving to a smaller subset $U''$, we
can assume that $\mc A_n$ is Dedekind-finite for each $n \in U''$.  We then move finally to a
yet smaller subset to finish the proof.
\end{proof}

Proposition~\ref{diconstant} and Corollary~\ref{altdiconstant} also hold for ultraproducts.
Indeed, given an ultrafilter $\mc U$, we only need the weaker condition that for each $N\in\mathbb N$
we have that $\{ n \in \mathbb{N} : C_{\text{DI}}(\mc A_n)\geq N \} \in\mc U$ (or for $C_{\text{DI}}'$).
Hence also, with $\mathcal{A}_n := (\ell^1(BC, \omega_n), \ast)$ for every $n$, we have that
$(\mc A_n)_{\mc U}$ is Dedekind-finite while each $\mc A_n$ is Dedekind-infinite.
The analogue of Proposition~\ref{prop:1} is the following.

\begin{Prop}\label{prop:di_ultra}
Let $(\mathcal{A}_n)$ be a sequence of unital Banach algebras, let $\mc U$ be
an ultrafilter, and suppose that there exists $K>0$ such that $\{ n \in \mathbb{N} : C_{\text{DI}}(\mathcal{A}_n) \leq K \} \in \mc U$.  Then $(\mathcal{A}_n)_{\mc U}$ is Dedekind-infinite.
\end{Prop}

Let us finish this section by indicating how we could have used {\L}o\'s's Theorem instead.  We need to encode the property of being Dedekind-finite (with suitable ``norm control'') into the language of Banach algebras.  In fact, it seems easier to work with Dedekind-infinite.
The idea of the following is extracted from the proofs of Theorem~\ref{dfpositive} and Proposition~\ref{diconstant}.

\begin{Lem}\label{lem:sentence_for_di}
For $n \in \mathbb{N}$, let $\varphi_n$ be the sentence
\[ \inf_{a\in B_n} \inf_{b\in B_n} \max\big(
\|ab-1\|, 1-\|ba-1\| \big). \]
Let $\mc A$ be a unital Banach algebra and let $\epsilon \in (0, 1/3n^2)$. If $\varphi_n^{\mc A} < \epsilon$ then $C_{\textrm{DI}}(\mc A)\leq n^2 / (1- \epsilon)$. Moreover, $\varphi_n^{\mc A}=0$ if and only if $C_{\textrm{DI}}(\mc A)\leq n^2$.
\end{Lem}
\begin{proof}
Let $n \in \mathbb{N}$ be fixed throughout the proof. Assume $\varphi_n^{\mc A} < \epsilon$. There are $a,b\in\mc A$ with $\|a\|, \|b\|\leq n$, $\|ab-1\|<\epsilon$ and $1-\|ba-1\|<\epsilon$, so that $\|ba-1\| > 1-\epsilon$.  Let $u:=ab$, then by the Carl Neumann series $u \in \inv(\mc A)$ with $\| 1-u^{-1} \| \leq \epsilon / (1-\epsilon)$ and $\| u^{-1} \| \leq 1 / (1-\epsilon)$. Let $a':= u^{-1}a$ and $b':= b$, then $a'b'=1$ and $b'a' = bu^{-1}a$. We \textit{claim} that $b'a' \neq 1$. For assume towards a contradiction that $b'a'=1$, so
\[ 1-\epsilon < \|ba-1\| \leq \| ba-bu^{-1}a\|
\leq \|a\| \|b\| \|1-u^{-1}\|
\leq n^2 \frac{\epsilon}{1-\epsilon}. \]
This is impossible because $\epsilon \in (0,1/3n^2)$, and so we conclude that $b'a' \neq 1$.  As $\|a'\| \|b'\| \leq \|a\| \|b\| (1-\epsilon)^{-1}$, we have that $C_{\textrm{DI}}(\mc A) \leq n^2 / (1-\epsilon)$.
 
For the second part of the statement suppose first $\varphi_n^{\mc A} = 0$. Clearly $\varphi_n^{\mc A} < \epsilon$, hence $C_{\textrm{DI}}(\mc A)\leq n^2 / (1- \epsilon)$ for each $\epsilon \in (0,1/3n^2)$ by the first part of the statement. Thus $C_{\textrm{DI}}(\mc A)\leq n^2$.  

Conversely, suppose $C_{\textrm{DI}}(\mc A)\leq n^2$. Let us fix $\epsilon \in (0,1)$. We can find $a,b\in\mc A$ with $ab=1, ba\not=1$ and $\|a\|\|b\|\leq (n+\epsilon)^2$.  By rescaling, we may suppose that $\|a\|=\|b\| \leq n+\epsilon$. As $ba \in \mc A$ is an idempotent not equal to $1$, by Remark \ref{rem:c_di_defn} we know that $\|ba-1\|\geq 1$.  
Set $a':=n (n+\epsilon)^{-1} a$ and $b':= n (n+\epsilon)^{-1} b$, so that $\|a'\| = \|b'\| \leq n$. Let us observe that
\begin{align*}
\|a'b'-1\| &= \left| \frac{n^2}{(n+\epsilon)^2} -1 \right| = \frac{\epsilon (2n + \epsilon)}{(n+\epsilon)^2}  <  \epsilon(2n + \epsilon), \\
\|b'a'-ba\| &\leq \left| \frac{n^2}{(n+\epsilon)^2} -1 \right| \|b\| \|a\| \leq \left| \frac{n^2}{(n+\epsilon)^2} -1 \right| (n + \epsilon)^2 = \epsilon (2n + \epsilon).
\end{align*}

Consequently $\|b'a'-1\| \geq \|ba-1 \| - \|b'a' - ba\| > 1-\epsilon(2n + \epsilon)$.  We conclude that $\varphi_n^{\mc A} < \epsilon(2n + \epsilon)$, and as $\epsilon \in (0,1)$ is arbitrary, we see that in fact $\varphi_n^{\mc A} = 0$.
\end{proof}

With the above result we can give alternative proofs to Theorems \ref{dfposultra} and \ref{prop:di_ultra}.

\begin{proof}[Second proof of Theorem~\ref{dfposultra}]
We prove the contrapositive. Assume therefore $(\mc A_i)_{\mc U}$ is Dedekind-infinite, thus there is an $n \in \mathbb{N}$ such that $C_{\textrm{DI}}((\mc A_i)_{\mc U})\leq n^2$. By Lemma~\ref{lem:sentence_for_di} this is equivalent to saying $\varphi_n^{(\mc A_i)_{\mc U}} =0$, which in turn is equivalent to $\lim_{i \rightarrow \mc U} \varphi_n^{\mc A_i} = 0$
by {\L}o\'s's Theorem.  Thus for a fixed $\epsilon \in (0,1/3n^2)$ we can find $U \in \mc U$ such that $\varphi_n^{\mc A_i} < \epsilon$ for each $i \in U$. Applying Lemma \ref{lem:sentence_for_di} again, we conclude that $C_{\textrm{DI}}(\mc A_i) \leq n^2 / (1-\epsilon)$ and hence $\mc A_i$ is Dedekind-infinite for each $i \in U$.
\end{proof}

\begin{proof}[A proof of Proposition~\ref{prop:di_ultra} with {\L}o\'s's Theorem]
	We may (and do) assume that $K \in \mathbb{N}$ and $U:= \{i\in\mathbb N : C_{\textrm{DI}}(\mc A_i) \leq K^2\}\in\mc U$. Then by Lemma \ref{lem:sentence_for_di} we see that $\varphi_K^{\mc A_i} =0$ for each $i \in U$. Hence by {\L}o\'s's Theorem we obtain $0= \lim_{i \rightarrow \mc U} \varphi_K^{\mc A_i} = \varphi_K^{(\mc A_i)_{\mc U}}$. In view of Lemma \ref{lem:sentence_for_di} again this is equivalent to $C_{\textrm{DI}}((\mc A_i)_{\mc U}) \leq K^2$, thus $(\mc A_i)_{\mc U}$ is Dedekind-infinite.
\end{proof}

\section{Proper infiniteness}\label{sec:prop_inf}

Recall that a Banach algebra $\mc A$ is \emph{properly infinite} when there exist
idempotents $p\sim 1$ and $q\sim 1$ which are orthogonal, $pq=0=qp$.

\subsection{When the asymptotic sequence algebra is properly infinite}
The idea of the following proposition originates in \cite{rordamfriis}.

\begin{Prop}\label{liftidempfromcorona}
Let $(\mathcal{A}_n)$ be a sequence of unital Banach algebras, and let $p,q \in \asy(\mathcal{A}_n)$ be idempotents with $[p,q]=0$. Then there exist idempotents $P,Q \in \ell^\infty(\mc A_n)$ with $p= \pi(P)$, $q=\pi(Q)$ and $[P,Q] =0$.
\end{Prop}

\begin{proof}
Exactly as in the proof of Theorem~\ref{dfpositive}, given an idempotent $p\in\asy(\mc A_n)$, 
we can find an idempotent $P=(p_n)\in\ell^\infty(\mc A_n)$ with $\pi(P)=p$.

Now let $Y \in \ell^\infty(\mc A_n)$ be such that $q= \pi(Y)$. Let $Z:= (1-P)Y(1-P) + PYP$, then from $[p,q]=0$ we obtain
\begin{align}
\pi(Z) = (1-p)q(1-p) + pqp = q-pq-qp+pqp+pqp = q-2pq+2p^2q=q.
\end{align}
It is clear that $[P,Z]=0$, so if we write $Z=(z_n)$, then $[p_n,z_n]=0$ for every $n \in \mathbb{N}$.

As $\pi(Z)=q$, we see that $q=q^2$ is equivalent to $\textstyle{\lim_{n \rightarrow \infty} \Vert z_n - z^2_n \Vert =0}$. Let $\mu_n := \Vert z_n - z^2_n \Vert$ for every $n \in \mathbb{N}$. There is $M$ such that for every $n \geq M$ we have $\mu_n  < 1/8$. In view of Proposition \ref{approxidemp}, for every $n \geq M$ there is an idempotent $q'_n \in \mathcal{A}_n$ with $\textstyle{\Vert z_n - q'_n \Vert \leq f_{\Vert z_n \Vert}(\mu_n) \leq f_{\Vert Z \Vert}(\mu_n)} \leq f_{\Vert Z \Vert}(1/8)$. Moreover, for every $n \in \mathbb{N}$ we have $[q'_n,y] =0$ whenever $y \in \mathcal{A}_n$ is such that $[z_n,y]=0$. In particular, $[q'_n,p_n]=0$ for all $n \geq M$. 

By continuity of $f_{\|Z\|}$, it follows that $\textstyle{\lim_{n \geq N} f_{\|Z\|}(\mu_n) =0}$; consequently $\lim_{n \geq N} \Vert z_n - q'_n \Vert =0$. 
For every $n \in \mathbb{N}$ we define
\begin{align}
q_n := \left\{
\begin{array}{l l}
q'_n & \quad \text{if  } n \geq M, \\
0 & \quad \text{otherwise.} \\
\end{array} \right.
\end{align}
Since $\Vert q'_n \Vert \leq \Vert q'_n -z_n \Vert + \Vert z_n \Vert \leq f_{\Vert Z \Vert}(1/8) + \Vert Z \Vert$ for all $n \geq M$, it follows that $Q:=(q_n)$ is an idempotent in $\ell^\infty(\mc A_n)$.  We observe that $q= \pi(Q)$ by $\textstyle{\lim_{n \geq M} \Vert z_n - q'_n \Vert =0}$. It is clear from the above that $[P,Q]=0$, thus concluding the proof.
\end{proof}

\begin{Thm}\label{positivepi}
Let $(\mathcal{A}_n)$ be a sequence of unital Banach algebras such that $\asy(\mathcal{A}_n)$ is properly infinite. Then there is an $N \in \mathbb{N}$ such that $\mathcal{A}_n$ is properly infinite for every $n \geq N$. 
\end{Thm}
\begin{proof}	
Let $p,q \in \asy(\mathcal{A}_n)$ be mutually orthogonal idempotents with $p,q \sim 1$. By Proposition \ref{liftidempfromcorona} there exist idempotents $P,Q \in \ell^\infty(\mc A_n)$ with $p= \pi(P)$, $q= \pi(Q)$ and $[P,Q]=0$. It follows from $p,q \sim 1$ that there exist $A= (a_n)$, $B= (b_n)$, $C= (c_n)$, $D= (d_n) \in \ell^\infty(\mc A_n)$ such that $1 = \pi(A) \pi(B)$, $p = \pi(B) \pi(A)$, and $1 = \pi(C) \pi(D)$, $q = \pi(D) \pi(C)$. Now let us pick $\delta \in (0,1)$ sufficiently small, depending on the norms of $A,B,C,D$. More precisely, we require $\delta \in (0,1)$ to satisfy
\begin{align}
\Vert A \Vert \Vert B \Vert(1- \delta)^{-1} \delta + \delta < 1/2, \quad \Vert C \Vert \Vert D \Vert (1- \delta)^{-1} \delta + \delta < 1/2.
\end{align}
Write $P=(p_n)$ and $Q= (q_n)$, then $pq =0 =qp$ is equivalent to $\lim_{n \rightarrow \infty} \Vert p_n q_n \Vert =0= \lim_{n \rightarrow \infty} \Vert q_n p_n \Vert$. So there is $M \in \mathbb{N}$ such that $\Vert p_n q_n \Vert, \Vert q_n p_n \Vert < \delta$ for every $n \geq M$, and since $p_n,q_n \in \mathcal{A}_n$ are commuting idempotents it follows that $p_nq_n$ is an idempotent of small norm,
so $p_n q_n =0$; similarly $q_n p_n=0$.

The aim of the following is to show that the idempotents $p_n$ and $q_n$ are not only eventually orthogonal, but equivalent to the unit element $1_n$ of $\mathcal{A}_n$, eventually.

We observe that $1= \pi(A) \pi(B)$ is equivalent to $\textstyle{\lim_{n \rightarrow \infty} \Vert 1_n - a_n b_n \Vert =0}$, thus there is $M' \geq M$ such that $\Vert 1_n - a_n b_n \Vert < \delta$ for every $n \geq M'$. Consequently, by the Carl Neumann series $u_n:= a_n b_n \in \inv(\mathcal{A}_n)$, $\| u_n^{-1} \| <(1- \delta)^{-1}$ and $\Vert 1_n - u_n^{-1} \Vert < \delta (1- \delta)^{-1}$ for all $n \geq M'$. Thus we can define $p'_n := b_n u^{-1}_n a_n$ for all $n \geq M'$ and it is immediate that $p'_n \in \mathcal{A}_n$ is an idempotent with $p'_n \sim 1_n$.  We also have that $\textstyle{\sup_{n \geq M'} \Vert p'_n \Vert \leq \Vert A \Vert \Vert B \Vert (1- \delta)^{-1}}$.

Analogously, there is $M'' \geq M'$ such that $v_n:= c_n d_n \in \inv(\mathcal{A}_n)$, $\|v_n^{-1} \| < (1- \delta)^{-1}$ and $\Vert 1_n - v_n^{-1} \Vert < \delta (1- \delta)^{-1}$ for all $n \geq M''$.  Then define $q'_n := d_n v^{-1}_n c_n$ for all $n \geq M''$, so that $q_n'$ is an idempotent with $q'_n \sim 1_n$, and $\sup_{n \geq M''} \Vert q'_n \Vert \leq \Vert C \Vert \Vert D \Vert (1- \delta)^{-1}$.

The equality $p = \pi(B) \pi(A)$ is equivalent to saying that $\lim_{n \rightarrow \infty} \Vert p_n - b_n a_n \Vert =0$, and similarly $q = \pi(D) \pi(C)$ is equivalent to $\lim_{n \rightarrow \infty} \Vert q_n - d_n c_n \Vert =0$.  So there is $N \geq M''$ such that $\Vert p_n - b_n a_n \Vert < \delta$ and $\Vert q_n - d_n c_n \Vert < \delta$ whenever $n \geq N$.

For every $n \geq N$ we have
\begin{align}\label{closeidemps}
\Vert p'_n -p_n \Vert & \leq \Vert b_n u_n^{-1} a_n -b_n a_n \Vert + \Vert b_n a_n - p_n \Vert \notag \\
& \leq \Vert b_n \Vert \Vert u_n^{-1} -1_n \Vert \Vert a_n \Vert + \Vert b_n a_n - p_n \Vert \notag \\
& \leq \Vert A \Vert \Vert B \Vert (1- \delta)^{-1} \delta + \delta \notag \\
& < 1/2.
\end{align}
Therefore by Lemma \ref{zemanek} it follows that $p'_n \sim p_n$, and since $\sim$ is an equivalence relation on the set of idempotents of $\mathcal{A}_n$, we have $p_n \sim 1_n$. Similarly, we conclude $q_n \sim 1_n$ for $n \geq N$. Since $p_n$ and $q_n$ are orthogonal, the claim follows.
\end{proof}

\begin{Rem}
We see from the first paragraph of the proof of Theorem \ref{positivepi} that pairs of mutually orthogonal idempotents from $\asy(\mc A_n)$ can be lifted to mutually orthogonal idempotents in $\ell^{\infty}(\mc A_n)$. More precisely, if $p,q \in \asy(\mc A_n)$ are idempotents with $pq=0=qp$, then there exist idempotents $P, Q= \in \ell^{\infty}(\mc A_n)$ such that $p=\pi(P)$, $q= \pi(Q)$, and $PQ = 0 = QP$.
\end{Rem}

\subsection{An application to inductive limits of unital Banach algebras}
The construction of inductive limits of unital Banach algebras is given in \cite[Section~3.3]{blackadar1} and \cite[Section~1.3.4]{Palmer}, for
example. For us it will be enough to use the characterisation in terms of a universal
property.  Inductive limits seem to be more commonly considered in the setting of $C^*$-algebras
(where all connecting maps are contractions) or for locally convex spaces.  In the general Banach
algebra setting there are some subtleties, which we note below.

Let $I$ be a directed set and let $(\mc A_i)_{i\in I}$ be a family of unital Banach algebras.
We suppose that for $i\leq j$ there is a bounded unital homomorphism $\varphi_{j,i}:\mc A_i
\rightarrow \mc A_j$, called the \emph{compatibility morphism}, which satisfies that
$\varphi_{i,i}=\id_{\mathcal{A}_i}$ for each $i \in I$, $\varphi_{k,j} \circ \varphi_{j,i} = \varphi_{k,i}$ for
$i\leq j\leq k$, and $\limsup_{j\geq i} \|\varphi_{j,i}\|<\infty$ for each $i \in I$.  We remark that the construction
will still work under the weaker condition that for each $i$ there is $K_i$ with $\limsup_{j\geq i}
\|\varphi_{j,i}(a)\| \leq K_i\|a\|$ for $a\in \mc A_i$.  However, this is not equivalent to the
stronger condition for general directed sets $I$: the Uniform Boundedness Principle does not apply,
due to the use of $\limsup$ (this is erroneously claimed in \cite{blackadar1, Palmer}; the claim would
hold with $I=\mathbb N$).

The (Banach algebra) inductive limit $\mc A = \varinjlim \mathcal{A}_i$ is uniquely (up to
isometric isomorphism) characterised by the universal property that:
\begin{enumerate}
\item for each $i\in I$ there is a bounded unital algebra homomorphism $\varphi_i:\mc A_i
\rightarrow \mc A$ with $\|\varphi_i(a)\| \leq \limsup_{j\geq i} \|\varphi_{j,i}(a)\|$ for
$a\in A_i$;
\item for $i\leq j$ we have that $\varphi_i = \varphi_j \circ \varphi_{j,i}$;
\item if $\mc B$ is another unital Banach algebra with bounded unital algebra homomorphisms $\phi_i:\mc A_i\rightarrow\mc B$ with $\phi_i = \phi_j\circ\varphi_{j,i}$ for $i\leq j$, and with $\|\phi_i(a)\|\leq \limsup_{j\geq i} \|\varphi_{j,i}(a)\|$ for each $i\in I, a\in A_i$, then there is a
unique contractive unital algebra homomorphism $\phi:\mc A\rightarrow\mc B$ with $\phi\circ\varphi_i = \phi_i$
for each $i \in I$.
\end{enumerate}
These conditions then imply that in (1) we have equality: $\|\varphi_i(a)\| = \limsup_{j\geq i} \|\varphi_{j,i}(a)\|$ for $i \in I$ and $a\in A_i$.
The universal property, (3), in particular uniqueness of $\phi$, implies that the union of the images of the $\varphi_i$ are dense in $\varinjlim \mathcal{A}_i$.

We remark that without the rather explicit norm condition, we do not seem to obtain a universal
condition, at this level of generality.  If each $\mc A_i$ is a $C^*$-algebra with each
compatibility morphism a $*$-homomorphism, then $\mc A$ is a $C^*$-algebra, and each compatibility
morphism is a contraction.  Let now $\mc B$ be another $C^*$-algebra with $*$-homomorphisms
$\phi_i: \mc A_i\rightarrow \mc B$ with $\phi_i = \phi_j\circ\varphi_{j,i}$ for each $i\leq j$.
Then for $a\in\mc A_i$ we have $\|\phi_i(a)\| = \|\phi_j(\varphi_{j,i}(a))\|
\leq \|\varphi_{j,i}(a)\|$ for each $j\geq i$, and so the norm condition is automatic in this
situation.

\begin{Prop}\label{injlimasy}
Let $( (\mathcal{A}_n), (\varphi_{i,j}) )$ be an inductive system of unital Banach algebras,
indexed by $\mathbb N$.  There is an isometric unital algebra homomorphism
$\theta: \varinjlim \mathcal{A}_n \rightarrow \asy(\mathcal{A}_n)$.
\end{Prop}
\begin{proof}
We use the universal property with $\mc B = \asy(\mathcal{A}_n)$.  Denote by $\pi:
\ell^\infty(\mc A_n)\rightarrow\mc B$ the natural quotient map.
For each $n$ define $\phi_n:\mc A_n\rightarrow\mc B$ by
\[ \phi_n(a) = \pi\big( \underbrace{0, 0,  \ldots, 0}_{(n-1) \text{ terms}},
a, \varphi_{n+1,n}(a), \ldots, \varphi_{i,n}(a), \ldots \big). \]
It is easy to see that the family $(\phi_n)$ satisfies the required commutation relations.
Further, by the definition of the norm on $\mc B$, we have that $\|\phi_n(a)\|
= \limsup_{k\geq n} \|\varphi_{k,n}(a)\|$ for all $n \in \mathbb{N}$, which implies the required norm relation.
There is hence a unital contractive homomorphism $\theta:\varinjlim \mathcal{A}_n \rightarrow
\asy(\mathcal{A}_n)$ with $\phi_n = \theta\circ\varphi_n$ for each $n \in \mathbb{N}$.
By our remark about condition (1), it follows that $\|\phi_n(a)\| = \|\varphi_n(a)\|$ for each $n\in\mathbb N$ and $a\in \mc A_n$.  Thus $\theta$ is actually isometric on the image of $\varphi_n$, and as the union of such images is dense in $\varinjlim \mathcal{A}_n$, it follows that $\theta$ is isometric, as claimed.
\end{proof}

The following lemma is straightforward.

\begin{Lem}\label{propinfpres}
Let $\mathcal{A}, \mathcal{B}$ be unital algebras and let $\psi: \, \mathcal{A} \rightarrow
\mathcal{B}$ be an algebra homomorphism which preserves the unit. If $\mathcal{A}$ is properly
infinite, then so is $\mathcal{B}$.
\end{Lem}

\begin{Cor}\label{indlim}
Let $(\mathcal{A}_n)$ be an inductive system of unital Banach algebras. If $\varinjlim \mathcal{A}_n$ is properly infinite then there is $N \in \mathbb{N}$ such that $\mathcal{A}_n$ is properly infinite for every $n \geq N$.
\end{Cor}
\begin{proof}
By Proposition \ref{injlimasy} there is an (contractive) algebra homomorphism $\theta: \, \varinjlim \mathcal{A}_n \rightarrow \asy(\mathcal{A}_n)$ which preserves the unit, hence by 
Lemma~\ref{propinfpres} the asymptotic sequence algebra $\asy(\mathcal{A}_n)$ is properly infinite.
The claim now follows from Theorem~\ref{positivepi}.
\end{proof}

\begin{Rem}
It is an unpublished observation of James Gabe that for $C^*$-algebras, Corollary \ref{indlim} follows from the semiprojectivity of the Cuntz algebra $\mathcal{O}_{\infty}$. We would like to thank him for communicating this result to us.
\end{Rem}

\subsection{When the sequence consists of properly infinite algebras}

We first demonstrate that the converse of Theorem~\ref{positivepi} is false in general.
For a unital Banach algebra $\mc A$ we define
\begin{align}
C_{\text{PI}}(\mathcal{A}):= \inf \lbrace \|a\| \|b\| \|c\| \|d\| : a,b,c,d\in\mathcal{A},
ab=1=cd, ad = 0 = cb \rbrace.
\end{align}
It will also be useful to define an auxiliary constant 
\begin{align}
C'_{\text{PI}}(\mathcal{A}):= \inf \lbrace \Vert p \Vert \Vert q \Vert : \; p,q \in \mathcal{A}, \, p^2=p,q^2=q, \, p \sim 1 \sim q, \, p \perp q \rbrace.
\end{align}
Notice that if we have $a,b,c,d$ as in the definition of $C_{\text{PI}}(\mathcal{A})$ then setting $p=ba, q=dc$ then $p^2 = baba = ba = p$ and similarly $q^2=q$, $pq=qp=0$, and $p\sim 1\sim q$ because $ab=1=cd$.  As then $\|p\| \|q\| \leq \|a\| \|b\| \|c\| \|d\|$ we see that $C'_{\text{PI}}(\mathcal A) \leq C_{\text{PI}}(\mathcal A)$.  If $\mathcal{A}$ is properly infinite then $1 \leq C_{\text{PI}}(\mathcal{A}) < + \infty$, otherwise $C_{\text{PI}}(\mathcal{A}) = + \infty$.  Clearly $C'_{\text{PI}}(\mathcal A) = +\infty$ if and only if $C_{\text{PI}}(\mathcal A) = +\infty$.  

As when we considered an algebra being Dedekind-infinite, the constant $C_{\text{PI}}'$ seems more
natural, but $C_{\text{PI}}$ seems more useful.  However, for being properly infinite, we shall
actually obtain a complete characterisation (see Proposition~\ref{piconstant} and
Proposition~\ref{piconstant_conv}) using $C_{\text{PI}}$.  Furthermore, Proposition~\ref{prop:7}
shows that $C_{\text{PI}}'$ and $C_{\text{PI}}$ are not comparable.

Notice that if $\mc A$ is properly infinite, then it is Dedekind-infinite, because if $p,q$
are orthogonal with $p\sim 1$ and $q\sim 1$, we cannot have $p=1$.

\begin{Lem}\label{lem:11}
Let $\mc A$ be properly infinite.  Then $C_{\text{DI}}(\mc A) \leq C_{\text{PI}}(\mc A)$
and $C'_{\text{DI}}(\mc A) \leq C'_{\text{PI}}(\mc A)$.
\end{Lem}
\begin{proof}
That $C'_{\text{DI}}(\mc A) \leq C'_{\text{PI}}(\mc A)$ is clear, given the remark before the
lemma.  Let $K>C_{\text{PI}}(\mc A)$ so we can find $a,b,c,d\in\mc A$ with
$\|a\| \|b\| \|c\| \|d\| \leq K$ and $ab=cd=1, badc = dcba = 0$.  Then $ba\not=1$, and so
$C_{\text{DI}}(\mc A) \leq \|a\| \|b\|$.  As $cd=1$ we have that $\|c\| \|d\|\geq 1$ and so
$\|a\| \|b\| \leq K$, from which the result follows.
\end{proof}

First we prove a slight strengthening of Theorem~\ref{positivepi}.

\begin{Prop}\label{piconstant}
Let $(\mathcal{A}_n)$ be a sequence of unital Banach algebras such that $\asy (\mathcal{A}_n)$ is properly infinite. Then there is a $K \geq 1$ and an $N \in \mathbb{N}$ such that $C_{\text{PI}}(\mathcal{A}_n) \leq K$ for all $n \geq N$.
\end{Prop}

\begin{proof}
The proof is a refinement of the proof of Theorem~\ref{positivepi}. We shall freely use the notation therein, and assume that all the objects have already been defined and the argument is repeated up until and including the inequality $\|p_n -p'_n \| <1/2$ (equation \eqref{closeidemps}), and the analogous inequality $\| q_n - q'_n \| <1/2$ for each $n \geq N$.

By Remark~\ref{rem:zemanek}, we can find $a_n', b_n', c_n', d_n' \in \mc A_n$ such that $p_n = a_n' b_n'$, $p_n' = b_n' a_n'$, $q_n = c_n' d_n'$, $q_n' = d_n' c_n'$, and the inequalities
\begin{align}
\|a_n'\|\|b_n'\| &\leq (4/3) \|p_n\|^2 (\|2p_n-1_n\|+ 1/2)^2 \leq (4/3) \|P\|^2 (2\|P\|+ 3/2)^2 \notag \\
\|c_n'\|\|d_n'\| &\leq (4/3) \|q_n\|^2 (\|2q_n-1_n\|+ 1/2)^2 \leq (4/3) \|Q\|^2 (2\|Q\|+ 3/2)^2
\end{align}
hold for each $n \geq N$.

Let us define $\tilde{a}_n:= a_n b_n'$, $\tilde{b}_n := a_n' b_n u_n^{-1}$, and $\tilde{c}_n:= c_n d_n'$, $\tilde{d}_n := c_n' d_n v_n^{-1}$ for each $n \geq N$. Thus we find
\begin{align}
\tilde{a}_n \tilde{b}_n &= a_n b_n' a_n' b_n u_n^{-1} = a_n p_n' b_n u_n^{-1} = a_n b_n u_n^{-1} a_n b_n u_n^{-1} = u_n u_n^{-1} u_n u_n^{-1} =1_n, \notag \\
\tilde{b}_n \tilde{a}_n &= a_n' b_n u_n^{-1} a_n b_n' = a_n' p_n' b_n' = a_n' b_n' a_n' b_n' = p_n p_n = p_n,
\end{align}
and similarly $\tilde{c}_n \tilde{d}_n = 1_n$, $\tilde{d}_n \tilde{c}_n = q_n$ for each $n \geq N$.  We also have the estimates
\begin{align}
\|\tilde{a}_n \| \| \tilde{b}_n \| &\leq 
\|a_n'\| \|b_n'\| \|a_n\| \|b_n\| \|u_n^{-1}\|
\leq 4(3- 3\delta)^{-1} \| A \| \|B \|  \|P\|^2 (2\|P\|+ 3/2)^2, \notag \\
\|\tilde{c}_n \| \| \tilde{d}_n \| &\leq 4(3- 3\delta)^{-1} \| C \| \|D \|  \|Q\|^2 (2\|Q\|+ 3/2)^2. 
\end{align}

Notice finally that $\tilde{b}_n \tilde{a}_n \tilde{d}_n \tilde{c}_n = p_n q_n = 0 = q_n p_n =  \tilde{d}_n \tilde{c}_n \tilde{b}_n \tilde{a}_n$, and hence $\tilde{a}_n \tilde{d}_n = 0 = \tilde{c}_n \tilde{b}_n$. Recalling that $\delta \in (0,1)$ depends only on the norms of $A,B,C$ and $D$ we conclude $\sup_{n \geq N} C_{\text{PI}}(\mathcal{A}_n) < + \infty$.
\end{proof}

We now aim to construct counter-examples to the converse of Theorem~\ref{positivepi}, for which we continue to
use semigroup algebras.  However, we now need to add a ``zero element''.

We say that $S$ is a \textit{monoid with a zero element} if $S$ is a monoid with at least two
elements and there exists a $\lozenge \in S$ such that $\lozenge s = \lozenge = s \lozenge$ for
all $s \in S$.  If such a $\lozenge \in S$ exists then it is necessarily unique.  As we assume that
$S$ has more than one element, we have that $\lozenge$ is different from the multiplicative
identity $e \in S$.

Let $\omega: \, S \rightarrow (0, + \infty)$ be a weight on $S$. Let $\mu:= \omega \vert_{S \setminus \lbrace \lozenge \rbrace}$, then $\mu: \, S \setminus \lbrace \lozenge \rbrace \rightarrow (0, + \infty)$ is such that $\mu(e)=1$. Every $\mu: \, S \setminus \lbrace \lozenge \rbrace \rightarrow (0, + \infty)$ arising in this way, as a restriction of a weight, will be referred to as a \textit{quasi-weight}.

We now explain how to define a product on the Banach space $\ell^1(S\setminus \{ \lozenge \}, \mu)$ (see also \cite[Section~3.2]{dlr} for a similar treatment).
This is accomplished by identifying $\ell^1(S\setminus \{ \lozenge \}, \mu)$ with the quotient
algebra $\ell^1(S, \omega) / \mathbb{C} \delta_{\lozenge}$.  With more details, we first notice
that $\mathbb{C} \delta_{\lozenge}$ is a closed two-sided ideal in $(\ell^1(S, \omega), \ast)$. Let $\pi: \, \ell^1(S, \omega) \rightarrow \ell^1(S, \omega) / \mathbb{C} \delta_{\lozenge}$ denote the quotient map. The symbol $\cdot$ will stand for the product on $\ell^1(S, \omega) / \mathbb{C} \delta_{\lozenge}$ induced by $\ast$. Let us consider the restriction map
\begin{align}
\psi: \; \ell^1(S, \omega) \rightarrow \ell^1(S\setminus \lbrace \lozenge \rbrace, \mu); \quad f \mapsto f \vert_{S \setminus \lbrace \lozenge \rbrace}.
\end{align}
This is a linear contractive surjection with $\Ker(\psi) = \mathbb{C} \delta_{\lozenge}$. 
Moreover, it also immediately follows from the definition that $\psi$ maps the open unit ball of $\ell^1(S, \omega)$ onto the open unit ball of $\ell^1(S\setminus \lbrace \lozenge \rbrace, \mu)$. Consequently, there is an isometric linear bijection
\begin{align}
\varphi: \; \ell^1(S, \omega) / \mathbb{C} \delta_{\lozenge} \rightarrow \ell^1(S\setminus \lbrace \lozenge \rbrace, \mu) 
\end{align}
which satisfies $\varphi \circ \pi = \psi$. This allows us to define a product on $\ell^1(S\setminus \lbrace \lozenge \rbrace, \mu)$ by setting
\begin{align}
f \# g := \varphi(\varphi^{-1}(f) \cdot \varphi^{-1}(g)) \quad (f,g \in \ell^1(S\setminus \lbrace \lozenge \rbrace, \mu)).
\end{align}
It is elementary to see that $\#$ is an algebra product on $\ell^1(S\setminus \lbrace \lozenge \rbrace, \mu))$. Furthermore, $(\ell^1(S\setminus \lbrace \lozenge \rbrace, \mu), \#)$ is a Banach algebra since $\Vert f \# g \Vert_{\mu} \leq \Vert f \Vert_{\mu} \Vert g \Vert_{\mu}$ holds for all $f,g \in \ell^1(S\setminus \lbrace \lozenge \rbrace, \mu)$ as the map $\varphi$ is an isometry.

For our purposes the most important property of $\ell^1(S\setminus \lbrace \lozenge \rbrace, \mu)$ is that for every $s,t \in S\setminus \lbrace \lozenge \rbrace$,
\begin{align}\label{likeconvbutitsnot}
\delta_s \# \delta_t = \left\{
\begin{array}{l l}
\delta_{st} & \quad \text{if  } st \neq \lozenge \\
0 & \quad \text{if } st = \lozenge. \\
\end{array} \right.
\end{align}
The above equality holds for the following reason. Observe that for $r \in S\setminus \lbrace \lozenge \rbrace$, we simply have $\delta_r = \delta_r \vert_{S\setminus \lbrace \lozenge \rbrace} = \psi(\delta_r)$. Consequently whenever $s,t \in S\setminus \lbrace \lozenge \rbrace$ then
\begin{align}
\varphi^{-1}(\delta_s \# \delta_t) &= \varphi^{-1}(\delta_s) \cdot \varphi^{-1}(\delta_t) = \varphi^{-1}(\psi(\delta_s)) \cdot \varphi^{-1}(\psi(\delta_t)) = \pi(\delta_s) \cdot \pi(\delta_t) \notag \\
&= \pi(\delta_s \ast \delta_t) = \pi(\delta_{st}) = \varphi^{-1}(\psi(\delta_{st})).
\end{align}
On the one hand if $st= \lozenge$ then $\psi(\delta_{st}) = \psi(\delta_{\lozenge}) = 0$. On the other hand if $st \neq \lozenge$ then $\psi(\delta_{st}) = \delta_{st}$, thus proving the claim.

In particular it follows from equation~\eqref{likeconvbutitsnot} that $(\ell^1(S\setminus \lbrace \lozenge \rbrace, \mu), \#)$ is a unital Banach algebra with $\delta_e$ being the unit, and such that $\Vert \delta_e \Vert_{\mu} = \mu(e) =1$.  The proof of the following is entirely analogous to that of Proposition~\ref{altpredfcounterex}.

\begin{Prop}\label{prepicounterex}
Let $S$ be a monoid with multiplicative identity $e \in S$ and a zero element $\lozenge \in S$. Let $\mu: \, S \setminus \lbrace \lozenge \rbrace \rightarrow [1, + \infty)$ be a quasi-weight on $S \setminus \lbrace \lozenge \rbrace$. Let $p \in (\ell^1(S \setminus \lbrace \lozenge \rbrace, \mu), \#)$ be a non-zero idempotent such that $p \neq \delta_e$. Then
\begin{align}
		\Vert p \Vert_{\mu} \geq \frac{1}{2} \inf \left\lbrace \mu(s): \, s \in S, \, s \neq e, \, s \neq \lozenge \right\rbrace.
\end{align}
\end{Prop}

In the following $\text{Cu}_2$ denotes the second Cuntz semigroup
\begin{align}
\langle a_1, a_2, b_1, b_2 : \; a_1 b_1 = e = a_2 b_2, \; a_1 b_2 = \lozenge = a_2 b_1 \rangle,
\end{align}
as defined in, for example, \cite[Page~141, Definition~2.2]{renault}. Here $e \in \text{Cu}_2$ and $\lozenge \in \text{Cu}_2$ denote the multiplicative identity and the zero element, respectively, rendering $\text{Cu}_2$ a monoid with a zero element.

Fix $n \in \mathbb{N}$. Let $\mu_n : \, \text{Cu}_2 \setminus \lbrace \lozenge \rbrace \rightarrow [1, + \infty)$ be a quasi-weight on $\text{Cu}_2 \setminus \lbrace \lozenge \rbrace$ defined as $\mu_n(e) = 1$ and $\mu_n(s) = n$ whenever $s \in \text{Cu}_2 \setminus \lbrace e, \lozenge \rbrace$.  Notice that this arises from the weight $\omega_n:\text{Cu}_2 \rightarrow [1,+\infty)$ defined by $\omega_n(e)=\omega_n(\lozenge) = 1$ and $\omega_n(s)=n$ otherwise.

\begin{Thm}\label{thm:pi_seq_not_pi}
Let $\mathcal{A}_n := (\ell^1(\text{Cu}_2 \setminus \lbrace \lozenge \rbrace, \mu_n), \#)$ for every $n \in \mathbb{N}$.
Then $(\mathcal{A}_n)$ is a sequence of properly infinite Banach algebras such that $\asy(\mathcal{A}_n)$ is not properly infinite.
\end{Thm}
\begin{proof}
Fix $n \in \mathbb{N}$. Let $p: = \delta_{b_1} \# \delta_{a_1}$ and $q: = \delta_{b_2} \# \delta_{a_2}$. Then $p,q \in \mathcal{A}_n$ are idempotents with $p \sim \delta_e \sim q$ and $p \perp q$ plainly because of the defining properties of $\text{Cu}_2$ and equation \eqref{likeconvbutitsnot}. This in particular shows that $\mathcal{A}_n$ is properly infinite. 

Let $p,q \in \mathcal{A}_n$ be arbitrary idempotents satisfying $p \sim \delta_e \sim q$ and $p \perp q$. Clearly $p,q \notin \lbrace \delta_e, 0 \rbrace$, hence Proposition \ref{prepicounterex} yields $\Vert p \Vert_{\mu_n}, \Vert q \Vert_{\mu_n} \geq n/2$, and consequently $C'_{\text{PI}}(\mathcal{A}_n) \geq n^2 / 4$. In view of Proposition~\ref{piconstant} the Banach algebra $\asy(\mathcal{A}_n)$ cannot be properly infinite. 
\end{proof}

Any reduced word in $\text{Cu}_2$ is of the form
$s = t_b s_a$, where $s_a$ is a word in $a_1, a_2$ (which are free, so generating a copy of
$\mathbb S_2$ the free semigroup on two generators), and $t_b$ is a word in $b_1, b_2$.
Consider how to cancel a word of the form
$s_a t_b$.  This will be equal to $\lozenge$ unless $s_a t_b = \cdots a_1^{n_3} a_2^{n_2}
a_1^{n_1} b_1^{n_1} b_2^{n_2} b_1^{n_3} \cdots$ with perhaps one of $s_a$ or $t_b$ having extra,
unbalanced, terms on the left, or right, respectively.  We can express this more succinctly as
follows.  Define $*$ to be the unique involution on $\text{Cu}_2$ with $a_i^*=b_i$ for $i=1,2$ and
$e^*=e, \lozenge^*=\lozenge$.  Notice that $r_b^* r_b = e$ for any word $r_b$.
Then $s_at_b=\lozenge$ unless either $s_a = r_a t_b^*$ or $t_b = s_a^* r_b$,
for some words $r_a$ and $r_b$.

From this, we can see that the idempotents in $\text{Cu}_2$ are of the form $s_b s_b^*$ for an
arbitrary word $s_b\in\mathbb S_2$.  Let $I(\text{Cu}_2)$ be the set of idempotents, \emph{excluding}
$\lozenge$.  One can also
show that if $s\in I(\text{Cu}_2), t\not\in I(\text{Cu}_2)$ then $st, ts \not\in I(\text{Cu}_2)$.
How idempotents multiply is a little more complicated.  Let $s = s_b s_b^*, t = t_b t_b^*$, and
consider $st$.  This will equal $\lozenge$ unless either:
\begin{itemize}
\item $t_b = s_b r_b$, in which case $st = s_b s_b^* s_b r_b r_b^* s_b^*
= s_b r_b r_b^* s_b^* = t$; or
\item $s_b^* = r_a t_b^*$, that is, $s_b = t_b r_b$ for some $r_b$, in which case
$st = t_b r_b r_b^* t_b^* t_b t_b^* = t_b r_b r_b^* t_b^* = s$.
\end{itemize}
This motivates defining $s_b \preceq t_b$ if $t_b = s_b r_b$ for some word $r_b$, that is,
$s_b$ is a \emph{prefix} of $t_b$.
Then $st=t$ when $s_b \preceq t_b$, and $st=s$ when $t_b \preceq s_b$, and $st=\lozenge$
otherwise.

We can consider $(\ell^1(I(\text{Cu}_2)), \#)$.  To ease notation, let $\mathbb S_2(b)$ be the set of words in $b_1,b_2$, with $\emptyset$ the empty word (the identity), so that a member of
$I(\text{Cu}_2)$ has the form $xx^*$ for some $x\in \mathbb S_2(b)$.  Let $e_x = \delta_{xx^*}$,
so $(e_x)$ is a basis for $\ell^1(I(\text{Cu}_2))$ and the product is
\[ e_x \# e_y = \begin{cases} e_x &\text{if } y \preceq x, \\
e_y &\text{if } x\preceq y, \\ 0 &\text{otherwise}. \end{cases} \]

\begin{Lem}
Let $\mc A = (\ell^1(I(\text{Cu}_2)), \#)$ and let $f = \sum_x f(xx^*) e_x\in\mc A$.  Then $f$ is an
idempotent if and only if:
\begin{enumerate}
\item $f(xx^*)\in\{-1,0,1\}$ for all $x$, and only finitely many are non-zero;
\item $\sum_{y \prec x} f(yy^*) \in \{0,1\}$ for each $x$.
\end{enumerate}
\end{Lem}
\begin{proof}
We have that $e_x \# e_y = e_z$ exactly when one of $x,y$ is equal to $z$, and the other is a
prefix of $z$. Throughout the rest of the proof we shall write $f(x)$ instead of $f(xx^*)$ for the sake of readibility. Thus
\[ f\# f = \sum_{x,y} f(x) f(y) (e_x \# e_y) 
= \sum_z \Big( \sum_{x\prec z} f(z)f(x) + \sum_{x\prec z} f(x) f(z)
+ f(z)^2 \Big) e_z. \]
Thus $f\# f = f$ if and only if
\begin{equation}
f(z) = f(z)^2 + 2f(z) \sum_{x\prec z} f(x) \qquad (z\in\mathbb S_2(b)).
\label{eq:2}
\end{equation}

If the two conditions hold, suppose $\sum_{x\prec z} f(x) = 0$.  Let $w = zb_1$ (say)
so $y \prec w$ exactly when $y=z$ or $y\prec z$.  Thus $\sum_{y\prec w} f(y)
= f(z) + \sum_{x\prec z} f(x) = f(z)$ so $f(z)\in\{0,1\}$, and hence (\ref{eq:2}) holds.
If $\sum_{x\prec z} f(x) = 1$ then similar reasoning shows that $f(z)+1 \in \{0,1\}$ so
$f(z)\not=1$, and hence (\ref{eq:2}) holds.

Conversely, we perform induction on the length of the word in $\mathbb S_2(b)$, again using
that if $z$ is a word of length $n+1$ then $z = yb_i$ say, for $y$ a word of length $n$, and
then $x\prec z$ if and only if $x=y$ or $x\prec y$.  The details are the same as in the proof
of Lemma~\ref{lem:six}.
\end{proof}

\begin{Prop}\label{prop:9}
Let $\mu$ be a quasi-weight on $\text{Cu}_2\setminus\{\lozenge\}$ such that $\mu\geq 1$ and
$\mu(s) \geq N$ for each $s\not\in I(\text{Cu}_2)$.  Then $C_{\text{PI}}(
\ell^1(\text{Cu}_2\setminus\{\lozenge\}, \mu), \#) \geq (N/86)^{1/3}$
\end{Prop}
\begin{proof}
Set $\mc A = (\ell^1(\text{Cu}_2\setminus\{\lozenge\}, \mu)$.
By Lemma~\ref{lem:11} it suffices to prove that $C_{DI}(\mc A) \geq (N/86)^{1/3}$.
To show this, we can follow almost exactly the strategy of the proof of Proposition~\ref{prop:8},
given the preliminary observations made above.
\end{proof}

\begin{Prop}\label{prop:7}
For each $n\geq 1$ there exists a quasi-weight $\mu_n$ on $\text{Cu}_2\setminus\{\lozenge\}$ so
that with $\mc A = (\ell^1(\text{Cu}_2 \setminus \{\lozenge\}, \mu_n), \#)$, we have
$C'_{\text{PI}}(\mc A)=1$ and yet $C_{\text{PI}}(\mc A_n) \geq n$.
\end{Prop}
\begin{proof}
We follow the strategy of the proof of Proposition~\ref{prop:2}.  Choose $N$ so that
$(N/86)^{1/3} \geq n$.  The set $X = \{ e, \lozenge, b_1a_1, b_2a_2 \}$ is a
sub-semigroup of $\text{Cu}_2$, and so the map $\omega:\text{Cu}_2\rightarrow [1,\infty)$ defined
by $\omega(s)=1$ for $s\in X$ and $\omega(s)=N\in\mathbb N$ otherwise, is a weight.  Let $\mu_n$
be the induced quasi-weight on $\text{Cu}_2\setminus\{\lozenge\}$.  With $p=\delta_{b_1a_1}$ and
$q=\delta_{b_2a_2}$, we see that $C'_{\text{PI}}(\mc A)=1$.  However, our quasi-weight satisfies
the hypotheses of Proposition~\ref{prop:9}, and so $C_{\text{PI}}(\mc A) \geq (N/86)^{1/3}
\geq n$.
\end{proof}

We can prove some similar renorming results.  The following is shown in exactly the same way
as Proposition~\ref{prop:10}, as if we have orthogonal idemopotents $p,q$ then $\{0,1,p,q\}$ is
a (bounded) semigroup in $\mc A$.

\begin{Prop}\label{prop:13}
Let $\mc A$ be a properly infinite Banach algebra.  There is an equivalent norm $\|\cdot\|_0$ on $\mc A$
such that $(\mc A, \|\cdot\|_0)$ is a unital Banach algebra, and
$C'_{\text{PI}}(\mc A, \|\cdot\|_0)=1$.
\end{Prop}

\begin{Prop}\label{prop:14}
For each $K>0$ there is a Banach algebra $\mc A$ and $a,b,c,d\in\mc A$ with $ab=cd=1,
cb = ad = 0$, such that, if $\|\cdot\|_0$ is any
equivalent norm on $\mc A$, then $\|a\|_0 \|b\|_0 \|c\|_0 \|d\|_0\geq K$.
\end{Prop}
\begin{proof}
We follow the strategy of the proof of Proposition~\ref{prop:11}.
We have the word-length $\ell$ on $\text{Cu}_2$, where $\ell(\lozenge) = 0$, and again this
is sub-additive.  Thus, for $x>0$, the function $\omega_x(s) = \exp(x\ell(s))$ is a weight.
Let $\mu$ be the quasi-weight given by $\omega$, and set
$\mc A = ((\ell^1(\text{Cu}_2 \setminus \{\lozenge\}, \mu_n), \#)$.  Let $\|\cdot\|_0$ be an
equivalent norm on $\mc A$, say with $m^{-1} \|f\|_0 \leq \|f\|_{\omega_x}$ for each $f\in\mc A$.

Set $a = \delta_{a_1}, b=\delta_{b_1}$ and $c = \delta_{a_2}, d=\delta_{b_2}$.
The same argument as used in  the proof of Proposition~\ref{prop:11} now shows that
$\|a\|_0, \|b\|_0, \|c\|_0, \|d\|_0 \geq e^x$, which completes the proof.
\end{proof}

Again, we leave open the question of whether it is possible to find an absolute constant $K \geq 1$ such that for every properly infinite Banach algebra $\mc A$ there is an equivalent norm $\| \cdot \|_0$ on $\mc A$ with $C_{\text{PI}}(\mc A, \|\cdot\|_0) \leq K$.

\subsection{When we have norm control}

As in the Dedekind-finite case, the converse to Theorem~\ref{positivepi} holds provided we have uniform norm control.  Notice that this, when combined with Proposition~\ref{piconstant}, gives a complete characterisation of when $\asy(\mathcal{A}_n)$ is properly infinite.

\begin{Prop}\label{piconstant_conv}
Let $(\mathcal{A}_n)$ be a sequence of unital Banach algebras such that \\
$\limsup_{n \rightarrow \infty} C_{\text{PI}}(\mathcal{A}_n) < \infty$.
Then $\asy(\mathcal{A}_n)$ is properly infinite.
\end{Prop}

We remark that this hypothesis is weaker than $\sup_n C_{\text{PI}}(\mathcal{A}_n) < \infty$, as the hypothesis of the proposition allows finitely many of the $\mathcal{A}_n$ to not be properly infinite.

\begin{proof}
By hypothesis, there are $K > 0$ and $N \in \mathbb{N}$ such that for $n\geq N$ we can find $a_n, b_n, c_n, d_n\in \mathcal{A}_n$ with $\|a_n\| \|b_n\| \|c_n\| \|d_n\| \leq K$ and so that, with $p_n = b_na_n, q_n = d_nc_n$, we have that $p_n, q_n$ are mutually orthogonal idempotents with $p_n \sim 1 \sim q_n$.  Notice that by rescaling, we may suppose that $\|a_n\| = \|b_n\|$ and $\|c_n\| = \|d_n\|$.  As $a_nb_n=1$, it follows that $\|a_n\|\geq 1$;
similarly $\|c_n\|\geq 1$.  Then $\|a_n\|^2 \|c_n\|^2 \leq K$ and so $\|a_n\|^2\leq K$ and
$\|c_n\|^2\leq K$.

For $n< N$ define $a_n=b_n=c_n=d_n=0$.  Then $A = (a_n) \in \ell^{\infty}(\mathcal{A}_n)$ with $\|A\|^2 \leq K$, and similarly for $B=(b_n), C=(c_n)$ and $D=(d_n)$.  We now see that $\pi(A) \pi(B) = 1$ in $\asy(\mathcal{A}_n)$, and similarly $\pi(C)\pi(D)=1$.  Furthermore, $p=\pi(B)\pi(A)$ and $q=\pi(D)\pi(C)$ are idempotents with $pq = qp = 0$.  Thus $p\sim 1 \sim q$
and $p,q$ are orthogonal, and so $\asy(\mathcal{A}_n)$ is properly infinite, as claimed.
\end{proof}

\begin{Cor}
Let $(\mathcal{A}_n)$ be a sequence of unital Banach algebras such that there is an $N \in \mathbb{N}$ such that $\mathcal{A}_n$ is properly infinite for all $n \geq N$. Moreover, suppose that one of the following two conditions hold:
\begin{enumerate}
\item $\mathcal{A}_n = \mathcal{A}_m$ for every $n,m \geq N$;
\item $\mathcal{A}_n$ is a $C^*$-algebra for each $n \in \mathbb{N}$.
\end{enumerate}
Then $\asy(\mathcal{A}_n)$ is properly infinite.
\end{Cor}
\begin{proof}
When $\mathcal{A}_n = \mathcal{A}_m$ for $n,m\geq N$, this follows immediately from the preceding result.  Now suppose that each $\mathcal{A}_n$ is a $C^*$-algebra.  From Proposition~\ref{prop:pr3}, a $C^*$-algebra $\mathcal{B}$ is properly infinite if and only if there are projections $p,q\in \mathcal{B}$ with $pq=0$ (so also $qp=0$) and with $p\approx 1 \approx q$.  In particular, $C_{\text{PI}}(\mathcal{B}) = 1$; and so the result follows again from the previous result.
\end{proof}

\subsection{For ultraproducts}
All of these results hold for ultraproducts with suitable modifications.  For example,
the analogue of combining Propositions~\ref{piconstant} and~\ref{piconstant_conv}
is the following.

\begin{Thm}
Let $(\mc A_n)$ be a sequence of Banach algebras, and let $\mc U$ be an ultrafilter.
Then $(\mc A_n)_{\mc U}$ is properly infinite if and only if there is $K>0$ such that
$\{ n \in \mathbb{N} : C_{\text{PI}}(\mc A_n)\leq K \} \in \mc U$.
\end{Thm}

We remark that again {\L}o\'s's Theorem could be used.  However, here the details of the analogue of Lemma~\ref{lem:sentence_for_di} seem complex, and we have chosen not to give them.

\section{Stable rank one}\label{sec:stab_rk1}

We say that a unital Banach algebra $\mc A$ has \emph{stable rank one} if the group of invertible elements $\inv(\mc A)$ is dense in
$\mc A$.  We recall, \cite[Proposition~3.1]{rsr}, that this is equivalent to either the left,
or the right, topological stable rank of $\mc A$ being $1$.  
We recall, see \cite[Lemma~2.1]{horvath1} for example, that having stable rank one implies
being Dedekind finite.  As shown in \cite[Example~2.2]{horvath1} the converse does not hold.
We hence view having stable rank one as a strict strengthening of being Dedekind-finite; and a
strengthening which is much studied for $C^*$-algebras.

\begin{Lem}\label{invsetsequal}
Let $(\mathcal{A}_n)$ be a sequence of unital Banach algebras. Then
\begin{align}
\inv \left( \asy(\mathcal{A}_n) \right) = \pi\Big( \inv\big(\ell^\infty(\mc A_n)\big) \Big).
\end{align}
\end{Lem}
\begin{proof}
For the non-trivial direction, let us pick an arbitrary $A = (a_n) \in \ell^\infty(\mc A_n)$
with $\pi(A)\in\inv(\asy(\mc A_n))$.  Thus there is $C=(c_n)\in\ell^\infty(\mc A_n)$ with
$\pi(A)\pi(C) = \pi(C)\pi(A) = 1$, that is, with
\[ \lim_{n \rightarrow \infty} \| c_n a_n - 1_n \| = \lim_{n \rightarrow \infty} \| a_n c_n - 1_n \| = 0. \]
Set $u_n=c_na_n$ and $v_n=a_nc_n$ for each $n$, so there is $N$ with $\|u_n-1_n\|<1/2$ and
$\|v_n-1_n\|<1/2$ for $n\geq N$.  Consequently, for $n\geq N$ we have that $u_n,v_n\in\inv(\mc A_n)$
with $\|u_n^{-1}\|,\|v_n^{-1}\| \leq 2$.  As $a_n u_n = a_n c_n a_n = v_n a_n$ for each $n$,
we have that $a_n u_n^{-1} = v_n^{-1} a_n$ for $n\geq N$.  Observe that
\[ a_n (c_n v_n^{-1}) = 1_n, \quad
(c_n v_n^{-1})a_n = c_n a_n u_n^{-1} = 1_n, \]
and so $a_n\in\inv(\mc A_n)$ with $a_n^{-1} = c_n v_n^{-1}$ and hence $\|a_n^{-1}\| \leq
2\|C\|$.  Define
\[ a_n' = \begin{cases} a_n &\text{if } n\geq N, \\ 1_n &\text{otherwise,} \end{cases}
\qquad
b_n = \begin{cases} a_n^{-1} &\text{if } n\geq N, \\ 1_n &\text{otherwise.} \end{cases} \]
Let $A'=(a_n'), B=(b_n)$ so that $A',B\in\ell^\infty(\mc A_n)$ and $A'B = BA'=1$, so that
$A'\in\inv(\ell^\infty(\mc A_n))$.  As $\pi(A) = \pi(A')$ the claimed result follows.
\end{proof}

\begin{Prop}\label{prop:15}
Let $\mc A$ be a unital Banach algebra such that $\asy(\mc A)$ has stable rank one.
Then also $\mc A$ has stable rank one.
\end{Prop}
\begin{proof}
If not, then there is $a\in\mc A$ and $\epsilon > 0$ with $\|a-b\| \geq \epsilon$ for each
$b\in\inv(\mc A)$.  Let $A=(a)\in\ell^\infty(\mc A)$. Since $\asy(\mc A)$ has stable rank one, there is $c\in\inv(\asy(\mc A))$ with $\|\pi(A)-c\|<\epsilon/2$.  By Lemma~\ref{invsetsequal}
there is $C=(c_n)\in\inv(\ell^\infty(\mc A))$ with $\pi(C)=c$, so that
\[ \epsilon/2 > \|\pi(A)-c\| = \|\pi(A)-\pi(C)\| = \limsup_{n \rightarrow \infty} \|a-c_n\|. \]
Hence there is some $n \in \mathbb{N}$ with $\|a-c_n\|<\epsilon$, and as each $c_n$ is invertible, this
gives the required contradiction.
\end{proof}

We now wish to improve this result, and completely characterise when $\asy(\mc A)$ has stable
rank one in terms of ``uniform'' approximation by invertibles for $\mc A$.
We give below, in Theorem~\ref{thm:6}, a counter-example to show that $\mc A$ can have stable
rank one while $\asy(\mc A)$ does not.  For $C^*$-algebras, this does always hold,
see Proposition~\ref{prop:6}.

\begin{Prop}\label{thm:5}
Let $\mc A$ be a unital Banach algebra.  The following are equivalent:
\begin{enumerate}
\item\label{thm:5:1}
There is a function $f:(0,\infty)\rightarrow\mathbb R$ such that for $\epsilon>0$
and $a\in\mc A$ with $\|a\|\leq 1$ there is $b\in\inv(\mc A)$ with $\|a-b\|<\epsilon$ and
$\|b^{-1}\|\leq f(\epsilon)$;
\item\label{thm:5:3} $\ell^\infty(\mc A)$ has stable rank one;
\item\label{thm:5:2} $\asy(\mc A)$ has stable rank one.
\end{enumerate}
\end{Prop}
\begin{proof}
Suppose $f$ exists.  Let $A=(a_n)\in\ell^\infty(\mc A)$.  By homogeneity we may suppose that
$\|a_n\|\leq 1$ for each $n$.  Given $\epsilon>0$, for each $n$ we can find $b_n\in\inv(\mc A)$
with $\|a_n-b_n\|<\epsilon$ and $\|b_n^{-1}\|\leq f(\epsilon)$.  Thus
$B=(b_n)\in\ell^\infty(\mc A)$ and $(b_n^{-1})$ is also in $\ell^\infty(\mc A)$.
So $B \in \inv(\ell^\infty(\mc A))$ and $\|A-B\|\leq \epsilon$.  As $\epsilon>0$ and $A$ were
arbitrary, this shows that $\ell^\infty(\mc A)$ has stable rank one.

If $\ell^\infty(\mc A)$ has stable rank one, also $\asy(\mc A)$ has stable rank one.

Now suppose that $\asy(\mc A)$ has stable rank one.
For $\epsilon > 0$ and $a\in\mc A$ with $\|a\|\leq 1$ let $I^{\epsilon}_a = \{b\in\inv(\mc A) : \|a-b\|<\epsilon\}$.
That $f$ exists is equivalent to showing that for each $\epsilon>0$,
\[ \sup \big\{ \inf\{ \|b^{-1}\|: b\in I^{\epsilon}_a\} : \|a\|\leq 1 \big\}< \infty. \]
Suppose this is not so.  Then there is $\epsilon > 0$ and a sequence $(a_n)$ with $\|a_n\|\leq 1$ for each $n \in \mathbb{N}$,
and with $\|b^{-1}\|\geq n$ for each $b\in I^{\epsilon}_{a_n}$.  Let $A :=(a_n)\in\ell^\infty(\mc A)$
and $a:=\pi(A)\in\asy(\mc A)$, so there is $c\in\inv(\asy(\mc A))$ with $\|a-c\|<\epsilon/2$.
Again, we can find $C=(c_n)\in\inv(\ell^\infty(\mc A))$ with $c=\pi(C)$. Thus in particular, there is $M>0$ with
$\|c_n^{-1}\|\leq M$ for all $n \in \mathbb{N}$. As $\limsup_{n \rightarrow \infty} \|a_n-c_n\|<\epsilon/2$, there is $N \in \mathbb{N}$ such that $\| a_n -c_n \| < \epsilon/2$ and hence $c_n \in I^{\epsilon}_{a_n}$ for each $n \geq N$. Then for any $n> \max(M,N)$ we obtain $\|c_n^{-1} \| \leq M \leq \max(M,N) < n \leq \|c_n^{-1} \|$, a contradiction. \end{proof}

We remark that it seems somewhat harder to characterise when $\asy(\mc A_n)$ has stable rank one,
for a sequence $(\mc A_n)$ of varying Banach algebras.  In the next section we develop some
results which allow us to say something about this more general situation.

\subsection{Stable rank one as a ``three space property''}
Having stable rank one is \emph{not} a three-space property (see \cite[Examples~4.13]{rsr}),
but in our special situation we can say something.
The following is the Banach-algebraic analogue of the ring-theoretic lemma \cite[Lemma~2.10]{horvath1}. We recall that if $\mathcal{A}$ is a unital algebra over a field $K$ with multiplicative identity $1_{\mathcal{A}}$, and $\mathcal{J} \trianglelefteq \mathcal{A}$ is a two-sided ideal, then $\tilde{\mathcal{J}}$ denotes the unital subalgebra $K1_{\mathcal{A}} + \mathcal{J}$. Moreover, $\inv(\tilde{\mathcal{J}}) = \inv(\mathcal{A}) \cap \tilde{\mathcal{J}}$ (see \cite[Lemma~2.4]{horvath1}).

\begin{Prop}\label{prop:threespsr1}
Let $\mathcal{A}$ be a unital Banach algebra and let $\mathcal{J} \trianglelefteq \mathcal{A}$ be a closed two-sided ideal such that both $\tilde{\mathcal{J}}$ and $\mathcal{A} / \mathcal{J}$ have stable rank one. Let $\pi: \, \mathcal{A} \rightarrow \mathcal{A} / \mathcal{J}$ denote the quotient map. If $\pi(\inv(\mathcal{A})) = \inv(\mathcal{A} / \mathcal{J})$ then $\mathcal{A}$ has stable rank one.
\end{Prop}
\begin{proof}
Let $a \in \mathcal{A}$ and $\epsilon > 0$ be arbitrary. Since $\mathcal{A} / \mathcal{J}$ has stable rank one, there is $c \in \mathcal{A}$ such that $\pi(c) \in \inv(\mathcal{A} / \mathcal{J})$ and $\Vert \pi(a) - \pi(c) \Vert < \epsilon /2$. By the assumption there is $d \in \inv(\mathcal{A})$ such that $\pi(d) = \pi(c)$ and thus $\Vert \pi(a) - \pi(d) \Vert < \epsilon/2$. Consequently there is $b \in \mathcal{J}$ such that $\Vert a - d - b \Vert < \epsilon /2$. Let us define $a':= b+d$. We observe that
\begin{align}
\pi(a'd^{-1}) = \pi(bd^{-1} +1) = \pi(b)\pi(d^{-1}) + \pi(1) = \pi(1),
\end{align}
equivalently, $1-a'd^{-1} \in \mathcal{J}$. This implies that $a'd^{-1} \in \tilde{\mathcal{J}}$. Now $\tilde{\mathcal{J}}$ has stable rank one, therefore we can pick $f \in \inv(\tilde{\mathcal{J}})=\inv(\mathcal{A}) \cap \tilde{\mathcal{J}}$ such that $\Vert a'd^{-1} - f \Vert < \epsilon / 2 \Vert d \Vert$. Clearly $fd \in \inv(\mathcal{A})$. Also,
\begin{align}
\Vert a - fd \Vert \leq \Vert a - a' \Vert + \Vert a' -fd \Vert \leq \Vert a - a' \Vert + \Vert a'd^{-1} -f \Vert \Vert d \Vert < \epsilon,
\end{align}
which shows that $\mathcal{A}$ has indeed stable rank one.
\end{proof}

\begin{Lem}\label{lem:5}
Let $(\mc A_n)$ be a sequence of unital Banach algebras all of which have stable rank one.
Let $\mc J = c_0(\mc A_n)$, considered as a closed, two-sided ideal in $\mc A=\ell^\infty(\mc A_n)$.
Then $\tilde{\mc J}$ has stable rank one.
\end{Lem}
\begin{proof}
This follows from \cite[Theorem~5.2]{rsr}, but we give the argument in this special case.
Let $A=(t1_n + a_n)\in\tilde{\mc J}$, so $\|a_n\|\rightarrow 0$ and $t\in\mathbb C$.  We wish
to approximate $A$ by a member of $\inv(\tilde{\mc J})$.  If $t=0$ then
pick $s\in\mathbb C$ non-zero and close to $t$.  If $A'=(s1_n+a_n)\in\tilde{\mc J}$ can approximated
by a member of $\inv(\tilde{\mc J})$ then so can $A$, because $A'$ is close to $A$.
So we may suppose that $t\not=0$.  If $t^{-1}A = (1_n + t^{-1}a_n)$ can be approximated by
a member of $\inv(\tilde{\mc J})$ then so can $A$.

So we may suppose that $t=1$.  Pick $\epsilon>0$
and choose $N$ so that $\|a_n\| < 1/2$ for $n\geq N$.
For $n\geq N$ let $c_n = -a_n + a_n^2 - a_n^3 + \cdots\in\mc A_n$, hence $\|c_n\| \leq \|a_n\|
(1-\|a_n\|)^{-1}$ and $c_n a_n = a_n c_n = -c_n - a_n$.  For $n<N$ use that $\mc A_n$ has stable
rank one to find $d_n\in\inv(\mc A_n)$ with $\|1_n + a_n - d_n\|\leq\epsilon$.  Set
$c_n = (d_n)^{-1} - 1_n$ for $n<N$.  Set $b_n=d_n-1_n$ for $n<N$ and $b_n = a_n$ for $n\geq N$.
Then $B=(b_n), C=(c_n)\in\mc J$.  Notice that
\[ (1_n+b_n)(1_n+c_n) = \begin{cases} d_n d_n^{-1} &\text{if } n<N, \\
1_n+a_n+c_n+a_nc_n &\text{if } n\geq N, \end{cases} \]
and so $(1_n+b_n)(1_n+c_n)=1_n$ for all $n$.
Similarly $(1_n+c_n)(1_n+b_n)=1_n$ for all $n$.  As $1+B, 1+C \in \tilde{\mc J}$ we
see that $1+B \in \inv(\tilde{\mc J})$.  Finally we consider $\| A-(1+B) \|$.
For $n<N$ we have that $\|(1_n+a_n) - (1_n+b_n)\| = \|1_n + a_n - d_n\|\leq\epsilon$,
while for $n\geq N$ we have that $(1_n+a_n) - (1_n+b_n) = a_n - a_n = 0$.
Hence $\| A-(1+B)\| \leq \epsilon$.
\end{proof}

\begin{Prop}\label{prop:4}
Let $(\mc A_n)$ be a sequence of unital Banach algebras all of which have stable rank one.
$\ell^\infty(\mc A_n)$ has stable rank one if and only if
$\asy(\mc A_n)$ has stable rank one.
\end{Prop}
\begin{proof}
If $\ell^\infty(\mc A_n)$ has stable rank one then clearly so does $\asy(\mc A_n)$.
Conversely, set $\mc A = \ell^\infty(\mc A_n)$ and $\mc J = c_0(\mc A_n)$ so that
$\asy(\mc A_n) = \mc A/\mc J$.
By Lemma~\ref{lem:5}, we see that $\tilde{\mc J}$ has stable rank one, and by
Lemma~\ref{invsetsequal} we know that $\inv(\mc A/\mc J) = \pi(\inv(\mc A))$.
Thus Proposition~\ref{prop:threespsr1} applies to show that $\mc A$ has
stable rank one.
\end{proof}

\subsection{For $C^*$-algebras}
We recall that in a $C^*$-algebra $\mathcal{A}$ an element $a \in \mathcal{A}$ has a \textit{unitary polar decomposition} if there exists a unitary $u \in \mathcal{A}$ such that $a = u \vert a \vert$.

\begin{Lem}\label{cstarpolar}
If $\mathcal{A}$ is a unital $C^*$-algebra such that every element of $\mathcal{A}$ has a unitary
polar decomposition then $\mathcal{A}$ has stable rank one.
\end{Lem}
\begin{proof}
Let $a \in \mathcal{A}$  and $\epsilon > 0$ be fixed. By the assumption there is a unitary $u \in \mathcal{A}$ such that $a = u \vert a \vert$.  By the Spectral Theorem, we know that $|a|+\epsilon 1
\in \inv(\mc A)$, hence also $b = u(|a|+\epsilon 1) \in\inv(\mc A)$.  Then $\|a-b\|
= \|\epsilon u\| = \epsilon$.  It follows that $\mc A$ has stable rank one.
\end{proof}

\begin{Prop}\label{prop:6}
Let $(\mc A_n)$ be a sequence of unital $C^*$-algebras having stable rank one.
Then $\asy(\mc A_n)$, and hence also $\ell^\infty(\mc A_n)$, have stable rank one.
\end{Prop}
\begin{proof}
This relies on an observation of Loring, \cite[Lemma~19.2.2]{loring}, which says that under this
hypothesis, every element of $\asy(\mc A_n)$ has a unitary polar decomposition.  The result
now follows from Lemma~\ref{cstarpolar} and Proposition~\ref{prop:4}.

For completeness, we give the short proof of \cite[Lemma~19.2.2]{loring}.  
Let $a\in\asy(\mc A_n)$ be $a=\pi(A)$ for some $A=(a_n)\in\ell^\infty(\mc A_n)$.  As for each
$n$ we have that $\inv(\mc A_n)$ is dense in $\mc A_n$, we can find $x_n\in\inv(\mc A_n)$
with $\lim_{n \rightarrow \infty} \|a_n - x_n\|=0$, so that $a = \pi(A')$ with $A'=(x_n) \in\ell^\infty(\mc A_n)$.
Notice that $(\|x_n^{-1}\|)$ might well be unbounded.  For each $n$ set
$u_n = x_n (x_n^*x_n)^{-1/2}$, which is a unitary in $\mc A_n$ with $u_n |x_n| = x_n$.  Then $U=(u_n)
\in\ell^\infty(\mc A_n)$ and $X=(|x_n|)\in \ell^\infty(\mc A_n)$ are such that $U$ is unitary
and $X = |A'|$, and $A' = UX$.  By uniqueness of positive square-roots, $\pi(X) = |a|$ and so
$\pi(U) |a| = a$ in $\asy(\mc A_n)$ is the required unitary polar decomposition.
\end{proof}

\begin{Rem}
This result, together with Proposition~\ref{thm:5}, shows that if $\mc A$ is a $C^*$-algebra
with stable rank one, then we get a form of uniform norm control on the approximating
invertible elements.  It would be interesting to know if this could be proved ``directly'',
in some sense.
\end{Rem}

\subsection{A counter-example}
We shall now present a construction which shows that Proposition~\ref{prop:6} does not hold
for Banach algebras.  

\begin{Thm}\label{thm:6}
The Banach algebra $\mc A = \ell^1(\mathbb Z)$, equipped with the convolution product, has stable
rank one.  For any non-principal ultrafilter $\mc U$ we have that $(\mc A)_{\mc U}$ does not
have stable rank one, and hence also $\asy(\mc A)$ and $\ell^\infty(\mc A)$ do not have
stable rank one.
\end{Thm}
\begin{proof}
Let $(p_n)$ be an increasing enumeration of the primes.  We shall first show that the ultraproduct
$(\ell^1(\mathbb Z/p_n\mathbb Z))_{\mc U}$ does not have stable rank one.  We do this by
collecting certain facts:
\begin{enumerate}
\item There is a contractive surjective homomorphism from
$(\ell^1(\mathbb Z/p_n \mathbb Z))_{\mc U}$ to $\ell^1(G)$ where $G$ is the set-theoretic
ultraproduct $(\mathbb Z/p_n\mathbb Z)_{\mc U}$.  This follows from \cite[Section~2.3.2]{dk}
following \cite[Section~5.4]{daw}.  Notice that $G$ is also
a commutative group.
\item The proof of \cite[Theorem~7.1]{ko} shows that $G$ is divisible and torsion-free and
that $G$ has cardinality the continuum.  It follows that there is a set $I$ of continuum
cardinality with $G$ and $\oplus_I \mathbb Q$ isomorphic as $\mathbb Q$-vector spaces, so certainly
isomorphic as abelian groups.  So $\ell^1(G)$ is Banach-algebra isomorphic to
$\ell^1(\oplus_I \mathbb Q)$.
\item Let $H = \oplus_I \mathbb Q$ and let $\widehat H$ be the dual group, a compact abelian
group.  By \cite[Proposition~4.14]{fol}, for example, we know that the Gel'fand tranform
(identified with the Fourier transform) gives a contractive homomorphism $\mc G:
\ell^1(\oplus_I \mathbb Q) \rightarrow C(\widehat H)$ which has dense range.
\item The compact space $\widehat H$ consists of all group homomorphisms $\oplus_I \mathbb Q
\rightarrow\mathbb T$, equipped with the topology of pointwise convergence.  It is easy to see
that this agrees with the compact space $(\widehat{\mathbb Q})^I$.  So $C(\widehat H)$ is
isomorphic with $C((\widehat{\mathbb Q})^I)$.
\item There is hence a dense range homomorphism $(\ell^1(\mathbb Z/p_n \mathbb Z))_{\mc U}
\rightarrow C((\widehat{\mathbb Q})^I)$.  
\item The compact abelian group $\widehat{\mathbb Q}$ is identified in \cite[Section~25.4]{hr}.
In particular, it is isomorphic to the ``$a$-adic solenoid'' $\Sigma_a$ for a suitable choice of
sequence $a$.  These compact groups are studied in \cite[Section~10]{hr}, and in particular,
\cite[Theorem~10.13]{hr} shows that $\Sigma_a$ is connected (and compact Hausdorff).
It follows from the definition, and \cite[Theorem~10.5]{hr}, that $\Sigma_a$ is a metrisable
space.
\item We now consider the \emph{covering dimension} of a topological space, see for example
\cite[Chapter~3]{p}.  In particular, it follows from \cite[Proposition~1.3]{p} that for a
Hausdorff space $X$, if $\dim(X)=0$ then $X$ is totally disconnected.
Thus $\dim(\widehat{\mathbb Q}) \geq 1$.
We shall also consider the \emph{small inductive dimension} of a topological space,
\cite[Chapter~4]{p}.  For a metric space, this is the same as the covering dimension,
\cite[Section~4, Chapter~5]{p}.  Finally, if $X$ is a compact metric space with $\dim(X)\geq 1$,
and $I$ an infinite set, then $X^I$ has infinite dimension.  This is shown for the small
inductive dimension in \cite[Example~1.5.17]{eng}, and hence also holds for the covering
dimension.
\item Rieffel's original motivation in \cite{rsr} was to generalise the covering dimension
to $C^*$-algebras (compare \cite[Theorem~1.1]{rsr} with \cite[Proposition~3.3.2]{p} for
example).  In particular, \cite[Proposition~1.7]{rsr} shows that if $X$ is a compact
(Hausdorff) space then the topological stable rank of $C(X)$ is $\lfloor \dim(X)/2 \rfloor + 1$.
\item In particular, $\widehat{\mathbb Q}^I$ has infinite dimension.  It
follows that $C(\widehat{\mathbb Q}^I)$ does not have stable rank one.  Hence also
$(\ell^1(\mathbb Z/p_n \mathbb Z))_{\mc U}$ does not have stable rank one.
\end{enumerate}

As $\ell^1(\mathbb Z) \rightarrow \ell^1(\mathbb Z/p_n \mathbb Z)$ is a quotient map,
for each $n$, it follows that $(\mc A)_{\mc U} \rightarrow
(\ell^1(\mathbb Z/p_n \mathbb Z))_{\mc U}$ is a quotient map, and so $(\mc A)_{\mc U}$
does not have stable rank one.  As $\asy(\mc A)$ quotients onto $(\mc A)_{\mc U}$,
also $\asy(\mc A)$ does not have stable rank one.  By Lemma~\ref{invsetsequal},
also $\ell^\infty(\mc A)$ does not have stable rank one.

Finally, we claim that $\mc A$ does have stable rank one.  This follows from the
more general result \cite[Corollary~1.6]{df}.  In fact, using that $\widehat{\mathbb Z}
= \mathbb T$ and that $C(\mathbb T)$ obviously has a dense set of invertibles, we can
instead appeal to \cite[Proposition~1.3]{df}.
\end{proof}

\begin{Rem}
The reader may wonder where the argument in the proof of Theorem \ref{thm:6} breaks if we attempt to apply it to group $C^*$-algebras instead of $\ell^1$ group algebras.  More concretely, let us consider the $\mathcal{A} :=C^*(\mathbb{Z})$, which is isomorphic as a $C^*$-algebra to $C(\mathbb{T})$.  As $\mathbb Z$ and $\mathbb Z / p_n\mathbb Z$ are discrete and amenable, the group
homomorphism $\mathbb Z \rightarrow \mathbb Z / p_n\mathbb Z$ induces a surjective $*$-homomorphism 
$\mc A \rightarrow C^*(\mathbb{Z} / p_n \mathbb{Z})$.  From \cite[Proposition~1.7]{rsr}, we see that $\mathcal{A}$ has stable rank one, and hence so does $C^*(\mathbb{Z} / p_n \mathbb{Z})$.  From Proposition \ref{prop:6} we know that $\asy(C^*(\mathbb{Z} / p_n \mathbb{Z}))$, and thus also $(C^*(\mathbb{Z} / p_n \mathbb{Z}))_{\mathcal{U}}$, has stable rank one.
	
On the other hand, when inspecting and adapting the reasoning in the proof of Theorem \ref{thm:6}, we find that there are $*$-isomorphisms between $C^*(G)$ and $C^*(\oplus_I \mathbb Q)$ and also with $C(\widehat{\mathbb Q}^I)$, where $G$ is the set-theoretic ultraproduct $(\mathbb Z/p_n\mathbb Z)_{\mc U}$.  As we have just seen, $C(\widehat{\mathbb Q}^I)$ does not have stable rank one, hence neither does $C^*(G)$. This means that there cannot be a continuous dense range algebra homomorphism from $(C^*(\mathbb{Z} / p_n \mathbb{Z}))_{\mathcal{U}}$ to $C^*(G)$, unlike for $\ell^1$ group algebras.
\end{Rem}

\subsection{For ultraproducts}

We quickly consider what happens when $\asy(\mc A_n)$ is replaced by an ultraproduct
$(\mc A_n)_{\mc U}$.  We first adapt Lemma~\ref{invsetsequal}.

\begin{Lem}\label{invsetsequal_up}
Let $(\mc A_n)$ be a sequence of unital Banach algebras, let $\mc U$ be an ultrafilter, and
denote by $\pi$ the quotient map $\ell^\infty(\mc A_n) \rightarrow (\mc A_n)_{\mc U}$.  Then
\[ \inv\big( (\mc A_n)_{\mc U} \big) = \pi\big( \inv(\ell^\infty(\mc A_n)) \big). \]
\end{Lem}
\begin{proof}
We can follow closely the proof of Lemma~\ref{invsetsequal}.  Let $a_n, c_n, u_n, v_n$ be as
before, and notice now that as $\pi(A)\pi(C) = \pi(C)\pi(A)=1$ there is some set $X\in\mc U$
with $\|u_n-1_n\|<1/2$ and $\|v_n-1_n\|<1/2$ for $n\in X$.  Then $a_n\in\inv(\mc A_n)$ and
$\|a_n^{-1}\|\leq 2\|C\|$ for each $n\in X$.  Hence we can find $A'\in\ell^\infty(\mc A_n)$
with $A'$ invertible and $\pi(A) = \pi(A')$.
\end{proof}

Proposition~\ref{prop:15} continues to hold, so if $(\mc A)_{\mc U}$ has stable rank one, then
so does $\mc A$.  Similarly, a close examination of the proof of Proposition~\ref{thm:5} shows
that it holds also for $(\mc A)_{\mc U}$.  In particular, we have:

\begin{Cor}
Let $\mc A$ be a Banach algebra.  Then $\asy(\mc A)$ has stable rank one if and only if
$(\mc A)_{\mc U}$ has stable rank one.
\end{Cor}

Finally, obviously then Proposition~6.9 shows that if $(\mc A_n)$ is a sequence of unital
$C^*$-algebras having stable rank one, then also $(\mc A_n)_{\mc U}$ also has stable rank one.
Theorem~\ref{thm:6} shows that this is not true for Banach algebras replacing $C^*$-algebras.

\section{Open questions}

We close the paper with some open questions.

\begin{itemize}
\item Does the analogue of Proposition~\ref{prop:10} hold for $C_{\text{DI}}$?
\item Does the analogue of Proposition~\ref{prop:13} hold for $C_{\text{PI}}$?
\item If $\asy(\mc A_n)$ has stable rank one, does $\mc A_n$ for large enough $n$?
\item Can one find a counter-example as in Theorem~\ref{thm:6} which uses directly the criteria
established in Proposition~\ref{thm:5}?
\end{itemize}

A wider ``open question'' is to study the constants $C_{\text{DI}}$ and $C_{\text{PI}}$, and
the criteria from Proposition~\ref{thm:5}.  These are ``metric versions'' of being Dedekind-infinite,
being properly infinite, and having stable rank one.  We wonder if there are other properties of
Banach algebras which have interesting ``metric versions''?

\begin{Ack}
The authors are grateful to Yemon Choi (Lancaster), James Gabe (Copenhagen), Saeed Ghasemi (Prague), Tomasz Kania (Prague) and Jared White (London) for many enlightening conversations. The authors are indebted to Niels J. Laustsen (Lancaster) for drawing their attention to the paper \cite{dlr}.  We are extremely grateful to the anonymous referee whose comments much improved the exposition of the paper.
The second-named author acknowledges with thanks the funding received from GA\v{C}R project 19-07129Y; RVO 67985840 (Czech Republic).
\end{Ack}

\Addresses

\end{document}